\documentclass[final]{siamltex}
\usepackage{graphicx, amssymb}
\usepackage{amsmath}
\usepackage{setspace}
\usepackage{fancyhdr}
\usepackage{amssymb}
\usepackage{bbold}
\usepackage{epsfig}
\usepackage[mathcal,mathscr]{eucal}
\usepackage{verbatim}
\usepackage{mathrsfs}
\newtheorem{thm}{Theorem}[section]

\newtheorem{lem}[thm]{Lemma}
\newtheorem{prpn}[thm]{Proposition}
\newtheorem{defn}[thm]{Definition}

\numberwithin{equation}{section}
\newcommand{\norm}[1]{\left\Vert#1\right\Vert}
\newcommand{\abs}[1]{\left\vert#1\right\vert}

\newcommand{\dbra}[1]{[\kern-0.15em[ #1 ]\kern-0.15em]}
\newcommand{\dbraco}[1]{[\kern-0.15em[ #1 [\kern-0.15em[}
\newcommand{\dbraoc}[1]{]\kern-0.15em] #1 ]\kern-0.15em]}

\title{On the Impulse Control of Jump Diffusions}
\author{Erhan Bayraktar  \thanks{Department of Mathematics, University of Michigan, 530 Church Street, Ann Arbor, MI 48109, USA.}
\and Thomas Emmerling \thanks{Department of Mathematics, University of Michigan, 530 Church Street, Ann Arbor, MI 48109, USA.}
\and Jos\'{e}-Luis Menaldi \thanks{Department of Mathematics, Wayne State University, 656 West Kirby Avenue, Detroit, MI 48202, USA.}}

\begin{document}

\maketitle

\begin{abstract}
Regularity of the impulse control problem for a non-degenerate $n$-dimensional jump diffusion with infinite activity and finite variation jumps was recently examined in \cite{DGW-2009}.  Here we extend the analysis to include infinite activity and infinite variation jumps.  More specifically, we show that the value function $u$ of the impulse control problem satisfies  $u\in W^{2,p}_{\mathrm{loc}}(\mathbb{R}^{n})$.        
\end{abstract}

\pagestyle{myheadings}
\thispagestyle{plain}
\markboth{ERHAN BAYRAKTAR, THOMAS EMMERLING, JOS\'{E}-LUIS MENALDI}{ON THE IMPULSE CONTROL OF JUMP DIFFUSIONS}

\section{Introduction} 
In this paper we analyze the regularity of the value function in an impulse control problem for an $n$-dimensional jump diffusion process.  We assume that the uncontrolled stochastic process $X$ is governed by the stochastic differential equation:
\begin{equation}
\label{sde}
\mathrm{d}X_{t}=\tilde{b}(X_{t-})\mathrm{d}t+\sigma(X_{t-})\mathrm{d}W_{t}+\int_{\mathbb{R}^{l}}j(X_{t-},z)\tilde{N}(\mathrm{d}t,\mathrm{d}z), \ X_{0}=x.
\end{equation}
Here $W$ is a $d$-dimensional standard Brownian Motion and $N$ is a Poisson random measure on $\mathbb{R}_{+}\times \mathbb{R}^{l}$, with $W$ and $N$ independent.  The L\'{e}vy measure $\nu(\cdot):=\mathbb{E}[N(1,\cdot)]$ may be unbounded and $\tilde{N}(\mathrm{d}t,\mathrm{d}z)$ is its compensated Poisson random measure with $\tilde{N}(\mathrm{d}t,\mathrm{d}z)=N(\mathrm{d}t,\mathrm{d}z)-\nu(\mathrm{d}z)\mathrm{d}t$.  Below, we specify the assumptions placed upon $\tilde{b},\sigma, j$ in order to ensure that the SDE is well-defined.  If an admissible control policy $V=(\tau_{1},\xi_{1};\tau_{2},\xi_{2}; \ldots)$ is chosen, then $X$ evolves as
\begin{equation}
\mathrm{d}X_{t}=\tilde{b}(X_{t-})\mathrm{d}t+\sigma(X_{t-})\mathrm{d}W_{t}+\int_{\mathbb{R}^{l}}j(X_{t-},z)\tilde{N}(\mathrm{d}t,\mathrm{d}z)+\sum_{i}\delta(t-\tau_{i})\xi_{i},
\end{equation}
where $\delta$ denotes the Dirac delta function.  Given a control $V:=(\tau_{1},\xi_{1}; \tau_{2},\xi_{2};\ldots)$, the objective function is 
\begin{equation}
\label{objfcn}
J_{x}[V]:=\mathbb{E}_{x}\left(\int_{0}^{\infty}e^{-rt}f(X_{t})\mathrm{d}t+ \sum_{i=1}^{\infty}e^{-r\tau_{i}}B(\xi_{i})\right).
\end{equation}
The goal is to minimize the objective function over all admissible control policies:
\begin{equation}
\label{vf}
u(x)=\inf_{V}J_{x}[V].
\end{equation}
Intuitively, we expect from the Dynamic Programming Principle that the value function $u(x)$ satisfies the following quasi-variational inequality
\begin{equation}
\label{HJB}
\tag{QVI}
\max\{-\mathcal{L}u+ru-f, u-\mathcal{M}u\}=0, \ x \in \mathbb{R}^{n},
\end{equation}
where $\mathcal{M}\varphi(x)$ is the minimal operator such that 
\begin{equation}
\mathcal{M}\varphi(x):=\inf_{\xi\in \mathbb{R}^{n}}(\varphi(x+\xi)+B(\xi)),
\end{equation}
and the partial integro-differential operator $\mathcal{L}$ is defined as
\begin{equation}
\begin{split}
\mathcal{L}\varphi(x)&:=\mathcal{L}_{D}\varphi(x)+I\varphi(x),
\end{split}
\end{equation}
with
\begin{equation}
\label{Aoperator}
\begin{split}
\mathcal{L}_{D}\varphi(x)&=\sum_{i,k=1}^{n}a_{ik}(x)\partial^{2}_{x_{i}x_{k}}\varphi(x)+\sum_{i=1}^{n}\tilde{b}_{i}(x)\partial_{x_{i}}\varphi(x),\\
I\varphi(x)&=\int_{\mathbb{R}^{l}}\left(\varphi(x+j(x,z))-\varphi(x)-j(x,z)\nabla\varphi(x)\right)\nu(\mathrm{d}z),
\end{split}
\end{equation}
where $(a_{ij})_{n\times n}:= \frac{1}{2}\sigma(x)\sigma(x)^{T}$.  \\   
\indent Analysis of the impulse control problem finds its roots in the classical works of \cite{BL-1982-2} and \cite{BL-1982}.  With regard to impulse control, these authors characterized the value function, analyzed optimal policies and discussed regularity of the value function in the non-degenerate diffusion case with bounded data.  Subsequent contributions such as \cite{M-1980-2}, \cite{M-1980-1}, \cite{M-1987} focused upon obtaining various characterizations of the value function for impulse control in more general settings than \cite{BL-1982-2} and \cite{BL-1982} such as the degenerate/non-degenerate pure/jump diffusion with bounded/unbounded data environments.  The focus of this paper is on identifying the regularity of the value function for impulse control under a general jump diffusion setting on the whole space and with unbounded controls.  Regularity in various relevent contexts has been examined by many in the literature, see e.g.  \cite{BX-2009}, \cite{BL-1982-2}, \cite{BL-1982}, \cite{DGW-2009}, \cite{GT-2001}, \cite{GL-1984}, \cite{GW-2009}, \cite{L-1982}, \cite{MR-1999}.  Recently, \cite{GW-2009} (resp. \cite{DGW-2009}) identified $W^{2,p}_{\mathrm{loc}}(\mathbb{R}^{n})$ regularity of the value function of impulse control for a pure diffusion (resp. jump diffusion) with unbounded controls.  In both of these papers, the authors utilized classical PDE arguments along with recent viscosity results for impulse control \cite{S-2009} to establish regularity.  For the jump diffusion case \cite{DGW-2009}, the authors establish $W^{2,p}_{\mathrm{loc}}(\mathbb{R}^{n})$ regularity for the value function for a jump process with finite variation jumps, i.e., integro-differential operators of order $[0,1]$.  With the regularity question resolved in this case, we examine whether this result can be leveraged to improve regularity to include jump processes which exhibit \emph{infinite} variation jumps, i.e., integro-differential operators of order $(1,2]$.  

We find, in Section~\ref{regularity-continuation-region}, that the regularity presented in \cite{DGW-2009} is particularly helpful in establishing regularity  in the continuation region $\mathcal{C}:=\{x\in \mathbb{R}^{n}: u(x)<\mathcal{M}u(x)\}$ for general jumps through approximation.  More specifically, we approximate the value function for the general jumps case using a value function for impulse control of a jump diffusion with finite variation jumps, i.e., integrable jumps $j^{\epsilon}(x,z)\in L^{1}(\mathbb{R}^{l}, \nu)$.  This value function converges uniformly on $\mathbb{R}^{n}$ (see Lemma \ref{UniformConv}) to the value function for infinite variation jumps and is in $W^{2,p}_{\mathrm{loc}}(\mathcal{C})$ via a weak limit argument (see Lemma \ref{W2ploc}).  This argument utilizes a variation of the local estimates for the integro-differential operator found in \cite{BX-2009}, \cite{GM-2002}, and \cite{MR-1999} (see Proposition \ref{AlmostLocal}) which only is valid in the continuation region $\mathcal{C}$.  Similar to \cite{DGW-2009}, a bootstrap method allows us to improve regularity so that $u\in C^{2,\frac{2\alpha-\gamma}{2}}(\mathcal{C})$ (see Proposition \ref{uHolder-cont}).  

For finite variation jumps, the authors in \cite{DGW-2009} show how establishing regularity of $u$ in the continuation region $\mathcal{C}$ can be particularly helpful in improving the result to the whole space, i.e., proving $u\in W^{2,p}_{\mathrm{loc}}(\mathbb{R}^{n})$.  This is primarily due to the fact that minimizers of $\mathcal{M}u(x)$ translate $x$ into the continuation region.  With this in mind, upon obtaining regularity in the continuation region for general jumps, we next examine in Section~\ref{sec:bdddmap} whether the same techniques carried out in \cite{DGW-2009} can be applied to smoothly carry $W^{2,p}_{\mathrm{loc}}$-regularity over into the action region $\mathcal{A}:=\{x\in \mathbb{R}^{n}: u(x)=\mathcal{M}u(x)\}$.  More specifically, this involves an examination of a Dirichlet problem on a bounded open set with a non-local integro-differential operator.  Resources for the regularity of second order elliptic integro-differential problems include  \cite{BCI-2008},  \cite{GM-2002}, \cite{GL-1984}, \cite{L-1982}, \cite{MR-1999} among others.  However, Dirichlet problems on bounded sets in the infinite variation case generate a singularity at the boundary.  As the monograph \cite{GM-2002} shows in detail, unless one is willing to restrict the state space of the jump process or impose the condition that only finite variation jumps can take the process outside the boundary, regularity cannot be guaranteed.  In order to avoid both of these unappealing restrictions, we develop a new approach to obtain $W^{2,p}_{\mathrm{loc}}$-regularity in the whole space.   Rather than analyzing $u$ as a solution to a variational inequality (VI) in an arbitrary bounded open set $\mathcal{O}$ in $\mathbb{R}^{n}$ as demonstrated in \cite{DGW-2009}, we obtain in Section~\ref{sec:wspapp} a characterization of $u$ as a distributional solution to a quasi-variational inequality (QVI) in $\mathbb{R}^{n}$.  Upon doing so, we then proceed to show that the distribution $(-\mathcal{L_{D}}-I+r)u$ is in fact a locally bounded function on $\mathbb{R}^{n}$.  Using this knowledge, an application of local estimates (Proposition \ref{AlmostLocal}) allows us to conclude $W^{2,p}_{\mathrm{loc}}(\mathbb{R}^{n})$ at the end of Section \ref{wholespace}.  \\
\indent The paper is organized as follows.  Section \ref{assumptions} provides the assumptions for the problem.  Section \ref{Integro-properties} discusses some useful properties relating to the value function and integro-differential operator.  Section \ref{regularity-continuation-region} presents regularity of the value function in the continuation region.  Section \ref{wholespace} presents the main regularity result, Theorem~\ref{thm:main}.  An Appendix includes proofs of some technical results from Section \ref{Integro-properties}, \ref{regularity-continuation-region}.  

\section{Assumptions}
\label{assumptions}
We adopt the notation used in \cite{DGW-2009} for function spaces if not explicitly defined and present the following assumptions:  \\ 

\noindent \emph{Lipschitz coefficients/running cost:}  We assume that the drift, volatility and the jump amplitude (in the first variable) in (\ref{sde}) are Lipschitz continuous and have Lipschitz continuous first derivatives (denoted $\tilde{b}^{'},\sigma', j'$), i.e., there exists a positive constants $C_{\tilde{b}}, C_{\sigma}>0$ and a positive function $C_{j}(z)\in L^{q}(\mathbb{R}^{l},\nu)$ for $q=1,2,4$ such that for any $x,y \in \mathbb{R}^{n}, z\in \mathbb{R}^{l},$
\begin{equation}
\tag{H1}
\label{coeffAs}
\begin{split}
&|\tilde{b}(x)-\tilde{b}(y)|\leq C_{\tilde{b}}\abs{x-y},\ \abs{\sigma(x)-\sigma(y)}  \leq C_{\sigma}\abs{x-y},\\
&\abs{j(x,z)-j(y,z)}\leq C_{j}(z)\abs{x-y}, \ \text{and there exists} \ C>0 \ \text{such that}
\end{split}
\end{equation}
\begin{equation}
\tag{H2}
\label{coeffDerLip}
\begin{split}
\abs{\tilde{b}^{'}(x)-\tilde{b}^{'}(y)}^{2}+\abs{\sigma'(x)-\sigma'(y)}^{2}+\int_{\mathbb{R}^{l}}\abs{j'(x,z)-j'(y,z)}^{2}\nu(\mathrm{d}z)\leq C\abs{x-y}^{2},
\end{split}
\end{equation}
where $\tilde{b}: \mathbb{R}^{n}\rightarrow \mathbb{R}^{n}$, $\sigma:  \mathbb{R}^{n}\rightarrow \mathbb{R}^{n\times d}$, $j: \mathbb{R}^{n}\times \mathbb{R}^{l}\rightarrow \mathbb{R}^{n}$.  Assume the running cost $f\geq 0$ is Lipschitz continuous, i.e., there exists a constant $C_{f}>0$ such that 
\begin{equation}
\label{fLip}
\tag{H3}
|f(x)-f(y)|\leq C_{f}|x-y|, \ \forall x,y \in \mathbb{R}^{n}.
\end{equation}

\noindent \emph{Semiconcavity:} Suppose for every open ball $B_{r}(0)$ of radius $r>0$ centered at $0$ (or simply denoted $B_{r}$), there exists a constant $C_{r}>0$ such that the function
\begin{equation}
\label{fsemiconcave}
\tag{H4}
x\mapsto f(x)-C_{r}\abs{x}^{2} \ \text{is concave}.
\end{equation}  

\noindent \emph{Jump conditions:}  For the jump amplitude $j$ and the L\'{e}vy measure $\nu$, we assume there exists some positive measurable function $j_{0}(z)$ such that
\begin{equation}
\label{BdsBigjumps}
\tag{H5}
\begin{split}
&\abs{j(x,z)}\leq j_{0}(z), \  \int_{\{j_{0}(z)\geq 1\}}[j_{0}(z)]^{2}\nu(\mathrm{d}z)\leq C_{0}<\infty,\\
&\int_{\{j_{0}(z)<1\}}[j_{0}(z)]^{p} \nu(\mathrm{d}z) \leq C_{0}<\infty, \ \text{for any} \ p \geq \gamma, \gamma\in [1,2]. 
\end{split}
\end{equation}
Assume that $j(x,z)$ is continuously differentiable in $x$ for any fixed $z$ and for any $x,x'$ and $0\leq \theta \leq 1$, there exists a constant $c_{0}>0$ such that 
\begin{equation}
\tag{H6}
\label{GMassumption}
\begin{split}
c_{0}\abs{x-x'}\leq \abs{(x-x')+\theta(j(x,z)-j(x',z))}\leq c_{0}^{-1}\abs{x-x'}.
\end{split}
\end{equation}
In particular, the Jacobian of $x \to j(x,z)$ satisfies
\begin{equation}
\begin{split}
c_{1}^{-1}\leq \mathrm{det}[I_{d}+\nabla j(x,z)]\leq C_{1},
\end{split}
\end{equation}
for any $x,z$ and some constants $c_{1},C_{1}\geq 1$, where $I_{d}$ is the identity matrix in $\mathbb{R}^{n}$, $\nabla j(x,z)$ is the matrix of the first partial derivatives in $x$, and $\mathrm{det}[\cdot]$ denotes the matrix determinant.  There exists a constant $M_{\gamma}>0$ such that
\begin{equation}
\tag{H7}
\label{adjointAs}
\begin{split}
&\abs{\nabla j(x,z)}\leq M_{\gamma}[j_{0}(z)]^{\gamma-1},\\
&\abs{\nabla \cdot j(x,z)-\nabla \cdot j(x+j(x,z),z)}\leq M_{\gamma}[j_{0}(z)]^{\gamma},
\end{split}
\end{equation}
where $\nabla\cdot j(x,z)$ denotes the divergence of the function $x\mapsto j(x,z)$ for any fixed $z$.  \\

\noindent \emph{Uniform Ellipticity.}  The diffusion component of $X$ satisfies the uniform ellipticity condition, i.e., there exists $\lambda>0$ such that
\begin{equation}
\label{unifEllAs}
\tag{H8}
\sum_{i,j=1}^{n}\xi_{i}a_{ij}(x)\xi_{j}\geq \lambda \abs{\xi}^{2}; \ \lambda>0, \ x \in \mathbb{R}^{n}.
\end{equation}

\noindent \emph{Transaction Cost:} The transaction cost function $B:\mathbb{R}^{n}\rightarrow \mathbb{R}$ is lower semi-continuous and satisfies:
\begin{equation}
\label{transactioncost}
\tag{H9}
\begin{cases}
& \inf_{\xi\in \mathbb{R}^{n}}B(\xi)=K>0,\\
& B\in C(\mathbb{R}^{n} \backslash \{0\}),\\
& |B(\xi)|\rightarrow \infty, \ \text{as} \ |\xi|\rightarrow \infty,\\
& B(\xi_{1})+B(\xi_{2})\geq B(\xi_{1}+\xi_{2})+K, \ \forall \ \xi_{1}, \xi_{2} \in \mathbb{R}^{n},\\
%& B(\xi)\leq K(1+\abs{\xi}^{q}), \ \xi\in \mathbb{R}^{n}, \ q>0.
\end{cases}
\end{equation}

\noindent \emph{Discount Rate:}  Assume the discount rate $r$ is sufficiently large.\\

%i.e.,   
%\begin{equation}
%\label{rlowerbound}
%\tag{H10}
%r>\kappa\geq \beta\geq 0.
%\end{equation}

%The constant $\beta$ can be found from the Lipschitz conditions in (\ref{coeffAs}) and $\kappa$ can be found using the semiconcavity assumption in (\ref{fsemiconcave}).  Both %constants $\beta, \kappa$ are detailed in (\ref{betas}), (\ref{kappas}) respectively.  \\
\indent The nonlocal integro-differential operator can be written as 
\begin{equation}
\label{IDO}
\begin{split}
I\varphi(x)&:=\int_{\mathbb{R}^{l}}\left(\varphi(x+j(x,z))-\varphi(x)-j(x,z)\cdot \nabla \varphi(x)\mathbb{1}_{\{j_{0}(z)<1\}}\right)\nu(\mathrm{d}z),
\end{split}
\end{equation}
and the local differential operator has the form
\begin{equation}
\begin{split}
\mathcal{L}_{D}\varphi(x)&:= \sum_{i,k=1}^{n}a_{ik}(x)\partial^{2}_{x_{i}x_{k}}\varphi(x)+\sum_{i=1}^{n}b_{i}(x)\partial_{x_{i}} \varphi(x) ,\\
\end{split}
\end{equation}
where $b:=\tilde{b}-\int_{\mathbb{R}^{l}}j(x,z)\mathbb{1}_{\mathbb{R}^{l}\setminus \{j_{0}(z)<1\}}\nu(\mathrm{d}z)$.\\

\section{Some Technical Estimates}
\label{Integro-properties}
In this section, we discuss preliminary regularity results of $u$ and $\mathcal{M}u$, prove some useful properties of the non-local operator $I$ and give a local $L^{p}$ estimate.   

\begin{lem}
\label{u*lip}
The function $u(\cdot)$ is Lipschitz continuous with constant $C_{u}$.  Additionally, $\mathcal{M}u(\cdot)$ is Lipschitz continuous.  
\end{lem}
\begin{proof}
In the Appendix, we provide a proof that $u$ is Lipschitz continuous within our setup.  Lemma 3.3 of \cite{DGW-2009} provides a proof for our setup that $\mathcal{M}u$ is Lipschitz continuous.
\end{proof}

\begin{defn}
Let $B_{r}(x)$ denote the open ball of radius $r$ centered at $x$.  The outer $\eta$-neighborhood of $\Omega$ is defined as $\Omega^{\eta}:=\{x\in \mathbb{R}^{n}: x\in B_{\eta}(y) \ \text{for some} \ y\in \Omega\}$.  %Similarly, the inner $\eta$-neighborhood of $\Omega$ is defined as $\Omega_{\eta}:=\{x\in \Omega: \mathrm{dist}(x,\partial \Omega)>\eta\}$.
\end{defn}

\begin{lemma}
\label{IntSobEst}
($\varepsilon$-$L^{p}$-estimates) Let $\mathcal{O}$ be an open subset of $\mathbb{R}^{n}$ and suppose (\ref{BdsBigjumps}), and (\ref{GMassumption}) hold.  Then, for any given $\varepsilon>0$, there exists $C(\varepsilon)>0$ depending on $\varepsilon$, such that for smooth $\varphi$, Lipschitz on $\mathbb{R}^{n}$ with constant $C_{\varphi}$, we have for $1\leq p \leq \infty$,
\begin{equation}
\begin{split}
\norm{I\varphi}_{L^{p}(\mathcal{O})}\leq \varepsilon \norm{\varphi}_{W^{2,p}(\mathcal{O}^{\varepsilon})}+C(\varepsilon)C_{\varphi}.
\end{split}
\end{equation}
\end{lemma}

\begin{proof}
See the Appendix.
\end{proof}

A direct application of Lemma \ref{IntSobEst} is the following local estimate for the integro-differential operator (see e.g. Proposition 2.4 in \cite{MR-1999}, Theorem 3.1.20 in \cite{GM-2002}, Proposition 3.5 in \cite{BX-2009}).  The estimate represents a direct extension of the classical $L^{p}$ interior estimates of Theorem 9.11 in \cite{GT-2001}.  
\begin{prpn}
\label{AlmostLocal}
(Local $L^{p}$-estimates) Suppose (\ref{coeffAs}), (\ref{BdsBigjumps}), (\ref{GMassumption}), and (\ref{unifEllAs}).  Let $\mathcal{O}'\subset \mathcal{O}$ be bounded open subsets of $\mathbb{R}^{n}$ with $\mathrm{dist}(\partial \mathcal{O}',\partial \mathcal{O})\geq \delta>0$.  Suppose that $v\in W^{2,p}_{\mathrm{loc}}(\mathcal{O})$, $v$ is Lipschitz on $\mathbb{R}^{n}$ with constant $C_{v}$, $1<p<\infty$.  Letting
\begin{equation}
(-\mathcal{L}_{D}-I+r)v=f \ \text{in} \ \mathcal{O},
\end{equation}
define the function $f$ in $\mathcal{O}$, there exists a constant $C$ depending on $n,p,\delta, \mathrm{diam}(\mathcal{O})$ and the bounds imposed by (\ref{coeffAs}) and (\ref{unifEllAs}) such that
\begin{equation}
\begin{split}
\norm{v}_{W^{2,p}(\mathcal{O}')}\leq C(\norm{f}_{L^{p}(\mathcal{O})}+C_{v}+\norm{v}_{L^{\infty}(\mathcal{O})}).
\end{split}
\end{equation}
\end{prpn}

\begin{proof}
This proof is similar to the proof of Proposition 3.5 in \cite{BX-2009}. For the sake of completeness we provided a proof in the Appendix.
\end{proof}

\begin{lem}
\label{HolderReg}
Assume (\ref{BdsBigjumps}) holds.  Suppose $\varphi$ is Lipschitz on $\mathbb{R}^{n}$ with constant $C_{\varphi}$.  Let $\Omega$ be a bounded open set of $\mathbb{R}^{n}$.  If $\varphi\in C^{1,\alpha}(\overline{\Omega^{1}})$ for some $\alpha\in[\gamma/2,1]$, then $I\varphi\in C^{0,\frac{2\alpha-\gamma}{2}}(\overline{\Omega})$ and
\begin{equation}
\begin{split}
\norm{I\varphi}_{C^{0,\frac{2\alpha-\gamma}{2}}(\overline{\Omega})}\leq C\left(C_{\varphi}+\norm{\varphi}_{C^{1,\alpha}(\overline{\Omega^{1}})}\right),
\end{split}
\end{equation}
for a positive constant $C$ dependent upon $\Omega, \alpha,\gamma$.
\end{lem}
\begin{proof}  This proof is similar to the proof of Lemma 3.2 in \cite{BX-2009}.  For details see the Appendix.
\end{proof}

\section{Regularity in the Continuation Region}
\label{regularity-continuation-region}
In this section, we establish the regularity of the value function $u$ in the continuation region $\mathcal{C}:=\{x\in \mathbb{R}^{n}: u(x)<\mathcal{M}u(x)\}$ through approximation.  As we show below, each approximate value function will satisfy the integrability assumption required in the regularity analysis undertaken in \cite{DGW-2009} and thus has $W^{2,p}_{\mathrm{loc}}$-regularity in $\mathbb{R}^{n}$.  Upon knowing this regularity for each approximation, we then show that a weak limit of the approximations exists by demonstrating that the sequence of solutions is bounded in $W^{2,p}_{\mathrm{loc}}$.  This argument utilizes the local $L^{p}$-estimates of Proposition \ref{AlmostLocal} and only holds in the continuation region $\mathcal{C}$.  To complete the argument, we then demonstrate that our sequence of approximations converges uniformly in $\mathbb{R}^{n}$ to $u$.  Finally, we implement a ``bootstrap" argument carried out in \cite{DGW-2009} to upgrade the regularity of $u$ in $\mathcal{C}$ to a H\"{o}lder space with two continuous derivatives.  We begin now with the approximation.

For $\epsilon>0$, set 
\begin{equation}
\label{j-epsilon}
\begin{split}
j^{\epsilon}(x,z):= j(x,z)\mathbb{1}_{\{j_{0}(z)>\epsilon\}}.
\end{split}
\end{equation}
With this definition, for each fixed $\epsilon>0$, it holds that $j^{\epsilon}\in L^{1}(\mathbb{R}^{l},\nu)$.  Indeed,
\begin{equation}
\begin{split}
\int_{\mathbb{R}^{l}}\abs{j^{\epsilon}(x,z)}\nu(\mathrm{d}z)\leq \int_{\{j_{0}>1\}}j_{0}(z)\nu(\mathrm{d}z)+\frac{1}{\epsilon^{2}}\int_{\{j_{0}\leq 1\}}[j_{0}(z)]^{2}\nu(\mathrm{d}z)<\infty.
\end{split}
\end{equation}
Letting $u_{\epsilon}$ denote the value function corresponding to a jump function $j^{\epsilon}$, we have that $u_{\epsilon}$ is Lipschitz continuous for each $\epsilon>0$.
\begin{lem}
\label{u*epLip}
For each $\epsilon>0$, the value function $u_{\epsilon}$ is Lipschitz continuous in $\mathbb{R}^{n}$ with constant $C_{u}$, the Lipschitz constant for $u$.  
\end{lem} 
\begin{proof}
The proof proceeds directly as in Lemma \ref{u*lip} since $\abs{j^{\epsilon}(x,z)-j^{\epsilon}(y,z)}\leq \abs{j(x,z)-j(y,z)}$.
\end{proof}

At this point, the regularity analysis presented in \cite{DGW-2009} allows us to conclude $u_{\epsilon}\in W^{2,p}_{\mathrm{loc}}(\mathbb{R}^{n})$ for each fixed $\epsilon >0$.  The next goal is to show uniform convergence of $u_{\epsilon}$ to $u$.  In doing so, we utilize a general estimate obtained for solutions of jump diffusions (see e.g. Chapter 5 in \cite{M-2008}).  For this estimate, we define the norm
\begin{equation}
\begin{split}
\norm{h-h'}_{0,p}&:=\underset{t,x}\sup  \Bigg\{\left(\int_{\mathbb{R}^{l}}\abs{h(t,x,z)-h'(t,x,z)}^{p}\nu(\mathrm{d}z)\right)^{1/p}\Bigg \},
\end{split}
\end{equation}
for $p\geq 2$.  Additionally, set
\begin{equation}
\begin{split}
\Lambda_{0,p}(h-h'):=\norm{h-h'}_{0,2p}+\norm{h-h'}_{0,2}.
\end{split}
\end{equation}

\begin{lem}
\label{moment-est}
Assume (\ref{coeffAs}), and suppose $r$ is sufficiently large.  Fix $\epsilon>0$.  Letting $X_{t}$ be a solution to (\ref{sde}) using jump function $j$ with  $X_{0}=x_{0}$ and $X^{\epsilon}_{t}$ be a solution using jump function $j^{\epsilon}$ and $X^{\epsilon}_{0}=x_{0}$, we have for $\alpha>\beta$,
\begin{equation}
\label{est-02}
\begin{split}
\mathbb{E}\Bigg [\underset{0\leq s \leq t}\sup\abs{X_{s}-X^{\epsilon}_{s}}^{2} e^{-\alpha s}\Bigg ] \leq M \Lambda_{0,2}^{2}(j-j^{\epsilon}),
\end{split}
\end{equation}  
for every $t\geq 0$ and for some constants $C, M$ which depend only upon $\alpha>\beta$, the bounds on $\tilde{b},\sigma, j$ and the dimensions $n, d$.  
\end{lem}

\begin{proof}
See the Appendix.
\end{proof}

\begin{lem}
\label{UniformConv}
Assume (\ref{coeffAs}), (\ref{fLip}), and suppose $r$ is sufficiently large.  The value function $u_{\epsilon}$ corresponding to a jump function $j^{\epsilon}$ converges uniformly on $\mathbb{R}^{n}$ to $u$, i.e., $u_{\epsilon}\overset{\mathrm{unif}}\longrightarrow u$ on $\mathbb{R}^{n}$.  
\end{lem}

\begin{proof}
Fix $\epsilon>0$ and let $X_{t}$ denote a solution to (\ref{sde}) with initial value $X_{0}=x$ and let $X^{\epsilon}_{t}$ denote a solution to (\ref{sde}) with jump function $j^{\epsilon}$ and initial value $X^{\epsilon}_{0}=x$.  From Lemma \ref{moment-est} and Jensen's inequality, we know for $\alpha>\beta$,
\begin{equation}
\label{est-for-result}
\begin{split}
\mathbb{E}\Bigg[ \underset{0\leq s \leq t}\sup\abs{X_{s}-X^{\epsilon}_{s}}\Bigg]&\leq e^{\alpha t/2}M^{1/2}\Lambda_{0,2}(j-j^{\epsilon}).
\end{split}
\end{equation}
Fix a control $V$ and let $J^{\epsilon}_{x}[V]$ denote the objective function (\ref{objfcn}) under $X^{\epsilon}$.  Using (\ref{fLip}) and (\ref{est-for-result}), we find
\begin{equation}
\begin{split}
J_{x}[V] &\leq J^{\epsilon}_{x}[V]+\mathbb{E} \Bigg [ \int_{0}^{\infty}e^{-rs}\abs{f(X_{s})-f(X^{\epsilon}_{s})}\mathrm{d}s \Bigg ]\\
&\leq J^{\epsilon}_{x}[V]+C_{f} \int_{0}^{\infty}e^{-rs}\mathbb{E}[\abs{X_{s}-X^{\epsilon}_{s}}]\mathrm{d}s\\
&\leq J^{\epsilon}_{x}[V]+C_{f} M^{1/2}\Lambda_{0,2}(j-j^{\epsilon})\int_{0}^{\infty}e^{-(r-\alpha/2)s}\mathrm{d}s.
\end{split}
\end{equation}
The final integral in the last inequality converges since $r$ is sufficiently large.  Let $C(\epsilon)$ denote the last term in the last inequality above.  Taking infimum over all controls yields
\begin{equation}
\begin{split}
u(x)\leq u_{\epsilon}^{*}(x) +C(\epsilon),
\end{split}
\end{equation}
where $C(\epsilon)\downarrow 0$ as $\epsilon \downarrow 0$.  Exchanging the roles of $X_{t}$ and $X^{\epsilon}_{t}$ yields $u_{\epsilon}(x)\leq u(x) +C(\epsilon)$.  Since $C(\epsilon)$ is independent of $x$, the convergence is uniform. 
\end{proof}

\begin{lem}
\label{W2ploc}
Assume (\ref{coeffAs}), (\ref{unifEllAs}), (\ref{fLip}), and suppose $r$ is sufficiently large.  In the continuation region $\mathcal{C}$, we have $u\in W^{2,p}_{\mathrm{loc}}(\mathcal{C})$.   
\end{lem}
\begin{proof}
Let $B\subset \mathcal{C}$ be closed and bounded.  Let $\delta=\underset{B}\inf\{\mathcal{M}u(x)-u(x)\}>0$.  By Lemma \ref{UniformConv}, $u_{\epsilon}$ converges uniformly to $u$ on $\mathbb{R}^{n}$ which, in turn, implies $\mathcal{M}u_{\epsilon}$ converges uniformly to $\mathcal{M}u$ on $\mathbb{R}^{n}$.  Using this information, there exists a $\epsilon ' (\delta)>0$ such that for all $\epsilon\in (0,\epsilon'(\delta))$, it holds that $B\subset \{x\in \mathbb{R}^{n}: u_{\epsilon}(x)<\mathcal{M}u_{\epsilon}(x)\}$.  For an open set $\mathcal{O}\subset B$ and any $\epsilon\in(0,\epsilon'(\delta))$, the local estimate Proposition \ref{AlmostLocal} along with Lemmas \ref{u*epLip} and \ref{UniformConv} yield that $\norm{u_{\epsilon}}_{W^{2,p}(\mathcal{O})}\leq C$ for some constant $C$ independent of $\epsilon$.  Thus, a weak limit exists and must coincide with the value function $u$ due to Lemma \ref{UniformConv}.  Since $B$ was arbitrary, the proof is complete.  
\end{proof}

As in \cite{DGW-2009}, we can now use a ``bootstrap" argument to obtain further regularity of $u$ in $\mathcal{C}$.   

\begin{prpn}
\label{uHolder-cont}
Assume (\ref{coeffAs}), (\ref{BdsBigjumps}), (\ref{unifEllAs}), (\ref{fLip}), and suppose $r$ is sufficiently large.  For any compact subset $D\subset \mathcal{C}$ of the continuation region, the value function $u$ is in $C^{2,\frac{2\alpha-\gamma}{2}}(D)$ for any $\alpha\in [\gamma/2,1]$ and satisfies $(-\mathcal{L}_{D}-I+r)u -f =0$ in $\mathcal{C}$.
\end{prpn}
\begin{proof}
First, consider any compact set $D$ such that $\overline{D^{1}}\subset \mathcal{C}$.  From Lemma \ref{W2ploc}, $u\in W^{2,p}(D^{1})$ for $p\in (1,\infty)$ from which Sobolev imbedding implies $u\in C^{1,\alpha}(\overline{D^{1}})$ for any $\alpha\in (0,1)$. Using this result and applying Lemma \ref{HolderReg}, we know that $Iu\in C^{0,\frac{2\alpha-\gamma}{2}}(D)$ for $\alpha\in[\gamma/2,1]$.  We now have enough regularity to use the Schauder estimates to improve our results.  Indeed, for any open ball $B\subset D\subset \overline{D^{1}} \subset \mathcal{C}$, the solution $v$ of the following classical Dirichlet problem
 \begin{equation}
\begin{split}
\begin{cases}
(-\mathcal{L}_{D}+r)v(x) =f(x)+Iu(x)& \text{a.e.} \ x\in B,\\
v(x)=u(x) & x \in \partial B,
\end{cases}
\end{split}
\end{equation}
is in $C^{2,\frac{2\alpha-\gamma}{2}}(B)$ by the Schauder estimates since $f+Iu(x)\in C^{0,\frac{2\alpha-\gamma}{2}}(D)$.  Now, from classical uniqueness results of viscosity solutions as used in Lemma 5.4 in \cite{DGW-2009} (see also final paragraph in Theorem 5.5 in \cite{DGW-2009}), we conclude $v=u\in C^{2,\frac{2\alpha-\gamma}{2}}(B)$ for any open ball $B\subset D$.  The choice of a compact set $D$ such that $\overline{D^{1}}\subset \mathcal{C}$ was necessary in order to apply Lemma \ref{HolderReg}.  However, the outer $1$-neighborhood $\Omega^{1}$  appears there as a result of our choice of magnitude $1$ to separate large and small jumps.  If, instead, we take any $\epsilon\in(0,1)$ to separate jump behavior, we would reach an analogous conclusion $u\in C^{2,\frac{2\alpha-\gamma}{2}}(B)$ for any open ball $B\subset D$ where $\overline{D^{\epsilon}}\subset \mathcal{C}$.  Hence, we find $u\in C^{2,\frac{2\alpha-\gamma}{2}}(C)$ for any compact set $C\subset \mathcal{C}$ and satisfies $(-\mathcal{L}_{D}-I+r)u -f =0$ in $\mathcal{C}$.     
\end{proof}

%BEGIN REGULARITY IN R^{N} SECTION
\section{Regularity in $\mathbb{R}^{n}$}  
\label{wholespace}
In this section, we investigate the regularity of the value function $u$ on the whole space.  The authors in \cite{DGW-2009} examine the regularity of $u$ under two specific assumptions concerning the L\'{e}vy measure: $\nu$ is finite and $j(x,\cdot)\in L^{1}(\nu)$.  These two assumptions describe qualities of the L\'{e}vy kernel $M(x,\mathrm{d}\eta)$ where
 \begin{equation*}
 \begin{split}
 M(x,A):=\nu\{z:j(x,z)\in A\}, \ A \ \text {-Borel measurable subset in} \ \mathbb{R}^{n},
 \end{split}
 \end{equation*}
which, in turn, determine the order of integro-differential operator $I$ (see Definition 2.1.2 in \cite{GM-2002}).  The assumptions taken in \cite{DGW-2009} concern integro-differential operators of order $\leq 1$.  Such operators map smooth functions to smooth functions.  For example, Lemma 5.1 in \cite{DGW-2009} shows that $I$ maps Lipschitz functions to Lipschitz functions when $I$ has order $0$.  Additionally, when $j(x,\cdot)\in L^{1}(\nu)$, Lemma 3.2 in \cite{DGW-2009} shows that $I$ maps a Lipschitz function to a continuous function when $C_{j}(\cdot)$ is $\nu$-integrable.  Since the value function for impulse control $u$ is Lipschitz continuous, it is known that $Iu$ is at least a continuous function under either assumption on $M(x,\mathrm{d}\eta)$.  As the authors in \cite{DGW-2009} demonstrate, the continuity of $Iu$ allows for a regularity analysis as in the pure diffusion case after defining a new running cost function $\tilde{f}:=f+Iu$.  Under our assumptions on $M(x,\mathrm{d}\eta)$, it is not known a priori that $Iu$ is continuous for Lipschitz continuous $u$ (for a similar discussion see \cite{BX-2009}).  As such, we cannot define $\tilde{f}$ as in \cite{DGW-2009} and must directly deal with the integro-differential operator. 

\subsection{Bounded Domain Approach} \label{sec:bdddmap}
With an integro-differential operator $I$ of order $\leq 1$, the authors in \cite{DGW-2009} show $u\in W^{2,p}_{\mathrm{loc}}(\mathbb{R}^{n})$ by studying the regularity of an associated optimal stopping time problem for a pure diffusion on bounded open sets of $\mathbb{R}^{n}$ (see Section 6 in \cite{DGW-2009}).  With a general jump case considered here, it is natural to consider the possibility of a similar proof argument involving an optimal stopping time problem for jump diffusions on bounded open sets of $\mathbb{R}^{n}$. \\ 
\indent Through penalization, regularity of an associated optimal stopping problem in a bounded open set $\mathcal{O}$ arises from the regularity of a Dirichlet problem.  As such, we may first consider the existence, uniqueness and regularity of a solution of the following Dirichlet problem:
\begin{equation}
\label{Dirichlet}
\begin{split}
\begin{cases}
(-\mathcal{L}_{D}-I+r)v(x)=f(x), \ x\in\mathcal{O},\\
v(x)=u(x), \ x\in \mathbb{R}^{n}\setminus \mathcal{O}.
\end{cases}
\end{split}
\end{equation} 
Notice that the non-local character of $I$ requires that the solution $v$ be defined on the support of the L\'{e}vy kernel $M(x,\cdot)$, namely, $\mathbb{R}^{n}$.  Integro-differential problems as above have been extensively discussed in the literature (see e.g. \cite{GM-2002}, \cite{GL-1984}, \cite{L-1982}).  Recalling this analysis, when studying (\ref{Dirichlet}) with a integro-differential operator $I$ of order $(1,2]$, $W^{2,p}(\mathcal{O})$ solutions exist if an extra condition is placed upon jumps outside of $\mathcal{O}$ (see (\ref{GLcondition})).  In the absence of this modification, only variational solutions in $W^{1,p}(\mathcal{O})$ exist.  The lack of dependence upon the fixed bounded open set $\mathcal{O}$ for $I$ of order $\leq 1$ renders this approach useful for establishing the regularity of $u$.  In fact, such an argument would essentially be the same as the analysis undertaken in both \cite{DGW-2009} and \cite{M-1980-1}.  The existence of this extra condition upon jumps outside $\mathcal{O}$ for integro-differential operators of order $>1$ does not disqualify this method from helping to achieve regularity for an optimal stopping problem associated to impulse control.  Indeed, the extra jump condition (\ref{GLcondition}) might automatically be satisfied depending on the value of $\gamma$ taken in (\ref{BdsBigjumps}).  To see this, consider the following two-step problem associated to (\ref{Dirichlet}).
\begin{equation}
\label{onestep}
\begin{split} 
\begin{cases}
(-\mathcal{L}_{D}+r)z(x)=0, \ x\in\mathcal{O},\\
z(x)=u(x), \ x\in  \mathbb{R}^{n} \setminus \mathcal{O},
\end{cases}
\end{split}
\end{equation}
and
\begin{equation}
\label{twostep}
\begin{split}
\begin{cases}
(-\mathcal{L}_{D}-I+r)w(x)=f(x)+Iz(x), \ x\in\mathcal{O},\\
w(x)=0, \ x\in \mathbb{R}^{n} \setminus \mathcal{O}.
\end{cases}
\end{split}
\end{equation}
If solutions exist to each problem, then $v=z+w$ will solve (\ref{Dirichlet}).  Sufficient conditions to solve (\ref{onestep}) are well-known and can be found in \cite{GT-2001}.  For (\ref{twostep}), there is a unique solution $w\in W^{2,p}(\mathcal{O})$ (see Theorem III.3 in \cite{GL-1984} and Theorem 3.1.22 in \cite{GM-2002}) if 
\begin{equation}
\label{GLcondition}
\begin{split}
\sup_{x\in \overline{\mathcal{O}}}\int \mathbb{1}_{\mathbb{R}^{n}\setminus \overline{\mathcal{O}}}(x+j(x,z))\abs{j(x,z)}^{1+\alpha}\nu(\mathrm{d}z)<\infty,
\end{split}
\end{equation}
where $0<\alpha<1/n$ and if $f+Iz\in L^{p}(\mathcal{O})$ for $n<p<1/\alpha$.  The condition (\ref{GLcondition}) is satisfied if $\gamma\in [1,2]$ in (\ref{BdsBigjumps}) is taken to satisfy $0<\gamma-1<1/n$.  Thus, we might be able to pursue this technique for showing regularity under a restricted set of $\gamma$ values in $[1,2]$ which depend upon the dimension $n$.  Even if we are content with this restriction, we cannot conclude the existence of a unique solution $w\in W^{2,p}(\mathcal{O})$ until $Iz\in L^{p}(\mathcal{O})$ for $n<p<1/\alpha$ is justified.  Recalling the classical results of Corollary 9.18 in \cite{GT-2001}, we know that $z\in W^{2,p}_{\mathrm{loc}}(\mathcal{O})\cap C^{0}(\overline{\mathcal{O}})$ from which Sobolev embedding implies that $z\in C^{0,1}(K)$ for any compact $K \subset \mathcal{O}$.  Since $z=u$ on $\mathbb{R}^{n}\setminus \mathcal{O}$, we can conclude that $z$ is Lipschitz continuous on $\mathbb{R}^{n}$.  However, $z$ Lipschitz continuous on $\mathbb{R}^{n}$ does not guarantee that $Iz\in L^{p}(\mathcal{O})$.  Essentially, unless we know more regularity about the solution $z$ with Lipschitz boundary function $u$, we are unable to obtain a $W^{2,p}(\mathcal{O})$ solution to (\ref{twostep}).  Due to this complication and the additional restriction to $\gamma$ beyond (\ref{BdsBigjumps}), we instead pursue an analysis of an integro-differential problem on the whole space rather than on a bounded open set $\mathcal{O}$.     

\subsection{The Whole Space Approach}\label{sec:wspapp}
In this section, we establish the following main theorem.  
\begin{thm}\label{thm:main}
Let the assumptions of Section \ref{assumptions} hold.  The value function of impulse control $u$ has a weak derivative up to order $2$ in $L^{p}(\mathcal{O})$ for $1<p<\infty$ and any bounded open set $\mathcal{O}$, i.e, $u\in W^{2,p}_{\mathrm{loc}}(\mathbb{R}^{n})$.  %Thus, $u$ has a locally Lipschitz first derivative.  
\end{thm}

The subsections to follow pursue a proof of the above result.  In the first, we present a characterization of the value function $u$.  In the second, we discuss the semi-concavity of $u$ and $\mathcal{M}u$ which assists in establishing regularity in the third.  

The following function spaces will be useful in order to examine the regularity of the value function $u$ on $\mathbb{R}^{n}$.  Let  $B_{p}(\mathbb{R}^{n})$ denote the space of Borel measurable functions $h$ from $\mathbb{R}^{n}$ into $\mathbb{R}^{n}$ such that 
\begin{equation}
\norm{h}_{p}=\sup\{\abs{h(x)}(1+\abs{x}^{2})^{-p/2}:x\in \mathbb{R}^{n}\} < \infty.
\end{equation}
Let $C_{p}(\mathbb{R}^{n})$ denote the subspace of $B_{p}(\mathbb{R}^{n})$ composed of $p$-uniformly continuous functions, i.e., all functions $h$ which satisfy: for every $\epsilon>0$ there exists a $\delta=\delta(\epsilon,p)$ such that for any $x,x'\in \mathbb{R}^{n}$, we have
\begin{equation}
\begin{split}
\abs{h(x)-h(x')}\leq \epsilon(1+\abs{x}^{p}), \ \abs{x-x'}<\delta.
\end{split}
\end{equation}
Let $C^{+}_{p}(\mathbb{R}^{n})$ denote the class of all positive functions in $C_{p}(\mathbb{R}^{n})$.  

\subsubsection{QVI}
Let $A:=-\mathcal{L}_{D}-I+r$ as in (\ref{Aoperator}).
Following \cite{SV-1972}, for any functions $u,v\in B_{p}(\mathbb{R}^{n})$, we say
\begin{equation}
\begin{split}
&Au=v, \ \text{in} \ \mathbb{R}^{n} \ (\text{resp.} \ \leq) \ \text{if the process}\\
&Y_{t}=\int_{0}^{t}v(X_{s})e^{-rs}\mathrm{d}s+u(X_{t})e^{-rt}, t\geq 0,
\end{split}
\end{equation}
is a martingale (resp. submartingale), for every initial $x\in \mathbb{R}^{n}$.  The following proposition from \cite{M-1987} characterizes the value function for our impulse control problem $u$.  
\begin{prpn}
\label{QVI}
Assume (\ref{coeffAs}), (\ref{fLip}), (\ref{transactioncost}), and suppose $r$ is sufficiently large.  Then the quasi-variational inequality
\begin{equation}
\begin{split}
\begin{cases}
\hat{u}\in C^{+}_{p}(\mathbb{R}^{n})\\
A\hat{u}\leq f \ \text{in} \ \mathbb{R}^{n}, \hat{u}\leq \mathcal{M}\hat{u} \ \text{in} \ \mathbb{R}^{n}, \\
A\hat{u}=f \ \text{in} \ [\hat{u}<\mathcal{M}\hat{u}],
\end{cases}
\end{split}
\end{equation}
with $[\hat{u}<\mathcal{M}\hat{u}]$ denoting the set of points $x$ such that $\hat{u}(x)<\mathcal{M}\hat{u}(x)$ has one and only one solution, which is given explicitly as the optimal cost for impulse control $u$.\\
\end{prpn}
\indent We can also give $Au$ a meaning as a distribution. In fact using the Lipschitz continuity of $u$, (\ref{GMassumption}), and (\ref{adjointAs}) we can see that this distribution satisfies,
for any open set $\mathcal{O}$ in $\mathbb{R}^{n}$ and every test function $\varphi \in \mathcal{D}(\mathcal{O})$ (compactly supported infinitely differentiable functions),
\begin{equation}
\label{ADistn}
\begin{split}
\langle Au,\varphi \rangle &=\sum_{i,j=1}^{n}\int_{\mathcal{O}}a_{ij}(x)\partial_{x_{i}}u(x)\partial_{x_{j}}\varphi(x)\mathrm{d}x\\
&-\sum_{i=1}^{n}\int_{\mathcal{O}}\mu_{i}(x)\partial_{x_{i}}[u(x)]\varphi(x)\mathrm{d}x+\int_{\mathcal{O}}ru(x)\varphi(x)\mathrm{d}x\\
&-\int_{\mathcal{O}}u(y)\mathrm{d}y \times\int_{\{j_{0}<1\}}[\varphi(y-j^*(y,z))-\varphi(y)+\nabla \varphi(y)\cdot j^*(y,z)]m^*(y,z)\nu(\mathrm{d}z)\\
&-\int_{\mathcal{O}}u(y)\mathrm{d}y \times\int_{\{j_{0}\geq1\}}[\varphi(y-j^*(y,z))-\varphi(y)]m^*(y,z)\nu(\mathrm{d}z)\\
&-\int_{\mathcal{O}}u(y)\mathrm{d}y \times \left(\int_{\{j_{0}<1\}}[j(y,z)-j^*(y,z)m^*(y,z)]\nu(\mathrm{d}z) \right)\cdot\nabla \varphi(y)\\ 
&-\int_{\mathcal{O}}u(y)\varphi(y)\mathrm{d}y \\
& \times \left(\int_{\{j_{0}\geq1\}}[m^{*}(y,z)-1]\nu(\mathrm{d}z)+ \int_{\{j_{0}<1\}}[m^*(y,z)+\nabla \cdot j(y,z)-1]\nu(\mathrm{d}z)\right),
\end{split}
\end{equation}
with $\mu_{i}=b_{i}-\sum_{j=1}^{n}\partial_{x_{j}}[a_{ij}]$, $j^*(y,z)=j(x(y,z),z)$, $m^*(y,z)=\mathrm{det}(\partial x(y,z)/\partial y)$ and the change of variable $y=x+j(x,z)$ (c.f. Section 2.4 in \cite{GM-2002}).\\
\indent The following proposition shows that the value function $u$ is a distributional solution once it is a martingale solution as above.
\begin{prpn}
\label{MartingaleToDistn}
Let $u$ be the value function of impulse control under the assumptions of Section \ref{assumptions} and suppose $U$ is an open set in $\mathbb{R}^{n}$.  The property that $Y_{t}=\int_{0}^{t}f(X_{s})e^{-rs}\mathrm{d}s+u(X_{t})e^{-rt}$
is a submartingale (resp. martingale) for every initial $x\in U$ implies that $Au\leq f$ (\text{resp.} \ Au=f) in $\mathcal{D}'(U)$, i.e. the inequality (\text{resp.} equality) is satisfied in the distributional sense.
\end{prpn}

\begin{proof}
This proof follows the approach taken in Proposition 2.5 in \cite{LM-2008}.  Without loss of generality we can assume $U$ is bounded.  Indeed, suppose $U$ is an unbounded open set.  We wish to show that for $\varphi\in \mathcal{D}(U)$, $\varphi\geq 0$ that $\langle f-Au,\varphi \rangle \geq 0$.  Since $\varphi\in C^{\infty}_{c}(U)$ there exists some bounded $U_{\text{bdd}}\subset U$ such that $\text{spt}(\varphi)\subset U_{\text{bdd}}$.  If it holds that $\langle f-Au, \phi \rangle \geq 0$ for all $\phi \in \mathcal{D}(U_{\text{bdd}})$, $\phi \geq 0$, then it is, indeed,  true that $\langle f-Au,\varphi \rangle \geq 0$.  Thus, we will assume below that $U$ is a bounded open set.  \\
\indent Let $X^{0}_{t}$ denote a solution of (\ref{sde}) with $X_{0}=0$.  Define the stopping time $\tau_{U}^{x}:=\inf\{t\geq 0: X^{0}_{t}+x \notin U\}$.  Fix $x_{0}\in U$ and define a stopping time as $\tau_{U}:=\inf\{t\geq 0: \exists y\in B_{x_{0}}(a) \ \text{such that} \ X^{0}_{t}+y \notin U\}$.  Choose $a>0$ such that $B_{x_{0}}(2a)\subset U$.  For every $(x,y)\in (B_{x_{0}}(a/2),B_{0}(a/2))$, we have $\tau_{U}\leq \tau_{U}^{x-y}$.  By the submartingale property,
\begin{equation}
\begin{split}
\mathbb{E}\Bigg[u(X^{0}_{t\wedge \tau_{U}}+x-y)e^{-r(t\wedge \tau_{U})}+\int_{0}^{t\wedge \tau_{U}}f(X_{s}^{0}+x-y)e^{-rs}\Bigg]\geq u(x-y).
\end{split}
\end{equation}
Letting $(\eta_{n})_{n=1}^{\infty}$ denote the standard regularizing sequence, we have
\begin{equation}
\begin{split}
&\int_{\mathbb{R}^{n}}\mathbb{E}[u(X^{0}_{t\wedge \tau_{U}}+x-y)e^{-r(t\wedge \tau_{U})}]\eta_{n}(y)\mathrm{d}y\geq \int_{\mathbb{R}^{n}}u(x-y)\eta_{n}(y)\mathrm{d}y\\
&-\int_{\mathbb{R}^{n}}\Bigg(\mathbb{E}\Bigg[ \int_{0}^{t\wedge \tau_{U}}f(X_{s}^{0}+x-y)e^{-rs}\mathrm{d}y\Bigg]\Bigg)\eta_{n}(y)\mathrm{d}y.
\end{split}
\end{equation}
Via Fubini's theorem, we find
\begin{equation}
\begin{split}
\mathbb{E}[u*\eta_{n}(X^{0}_{t\wedge \tau_{U}}+x)e^{-r(t\wedge \tau_{U})}]\geq u*\eta_{n}(x)-\int_{\mathbb{R}^{n}}\mathbb{E}\Bigg[\int_{0}^{t\wedge \tau_{U}}f(X_{s}^{0}+x-y)e^{-rs}\mathrm{d}s\Bigg]\eta_{n}(y)\mathrm{d}y.
\end{split}
\end{equation}
Then, for every $t>0$,
\begin{equation}
\begin{split}
\frac{1}{t}\left(\mathbb{E}[u*\eta_{n}(X^{0}_{t\wedge \tau_{U}}+x)e^{-r(t\wedge \tau_{U})}]-u*\eta_{n}(x)\right)\geq-\int_{\mathbb{R}^{n}}\mathbb{E}\Bigg[\frac{1}{t}\int_{0}^{t\wedge \tau_{U}}f(X_{s}^{0}+x-y)e^{-rs}\mathrm{d}s\Bigg]\eta_{n}(y)\mathrm{d}y,
\end{split}
\end{equation}
which implies
\begin{equation}
\begin{split}
\mathbb{E}\Bigg[\frac{1}{t}\int_{0}^{t\wedge \tau_{U}}A(u*\eta_{n})(X^{0}_{s}+x)\mathrm{d}s\Bigg]\leq \int_{\mathbb{R}^{n}}\mathbb{E}\Bigg[\frac{1}{t}\int_{0}^{t\wedge \tau_{U}}f(X_{s}^{0}+x-y)e^{-rs}\mathrm{d}s\Bigg]\eta_{n}(y)\mathrm{d}y.
\end{split}
\end{equation}
Since $U$ is bounded, the bounded convergence theorem yields
\begin{equation}
\begin{split}
\mathbb{E}\Bigg[\lim_{t\downarrow 0}\frac{1}{t}\int_{0}^{t\wedge \tau_{U}}A(u*\eta_{n})(X^{0}_{s}+x)\mathrm{d}s\Bigg]&\leq \int_{\mathbb{R}^{n}}\mathbb{E}\Bigg[\lim_{t\downarrow 0}\frac{1}{t}\int_{0}^{t\wedge \tau_{U}}f(X_{s}^{0}+x-y)e^{-rs}\mathrm{d}s\Bigg]\eta_{n}(y)\mathrm{d}y, \\
\mathbb{E}\Bigg[\lim_{t\downarrow 0}\frac{1}{t}\int_{0}^{t}\mathbb{1}_{\{ \tau_{U}\geq s\}}A(u*\eta_{n})(X^{0}_{s}+x)\mathrm{d}s\Bigg]&\leq \int_{\mathbb{R}^{n}}\mathbb{E}\Bigg[\lim_{t\downarrow 0}\frac{1}{t}\int_{0}^{t}\mathbb{1}_{\{ \tau_{U}\geq s\}}f(X_{s}^{0}+x-y)e^{-rs}\mathrm{d}s\Bigg]\eta_{n}(y)\mathrm{d}y, 
\end{split}
\end{equation}
The mean value theorem now implies that $A(u*\eta_{n})(x)\leq (f*\eta_{n})(x)$ for all $x\in B_{x_{0}}(a/2)$.  Notice that for the value function $u$, we know $u*\eta_{n} \rightarrow u$ in $L^{p}(B_{x_{0}}(a/2))$ and $(\partial_{x_{i}}u)*\eta_{n}\to \partial_{x_{i}}u$ in $L^{p}(B_{x_{0}}(a/2))$ and for any $1<p<\infty$.  Using (\ref{ADistn}), it is straightforward to show that $\langle A(u*\eta_{n}),\varphi \rangle $ converges to $\langle Au, \varphi \rangle$ as $n\to \infty$ in $\mathcal{D}'(B_{x_{0}}(a/2))$.  Combining this fact with  $A(u*\eta_{n})(x)\leq (f*\eta_{n})(x)$ for all $x\in B_{x_{0}}(a/2)$ allows us to conclude that $Au(x)\leq f(x)$ in $\mathcal{D}'(B_{x_{0}}(a/2))$.  Since $x_{0}\in U$ was arbitrary, a partition of unity argument now shows $Au(x)\leq f(x)$ in $\mathcal{D}'(U)$.

%PARTITION OF UNITY ARGUMENT SPECIFICS
% \\ \indent Since $x_{0}\in U$ was arbitrary, the following partition of unity argument shows $Au(x)\leq f(x)$ in $\mathcal{D}'(U)$.  Consider a fixed $\varphi \in \mathcal{D}(U)$ where $\varphi \geq 0$.  Let $K\subset U$ be a compact set such that $\mathrm{spt}(\varphi)=K$.  Let $a:=\mathrm{dist}(K,\partial U)>0$.  Since $K$ is compact, there exists finitely many $x_{i}\in K, 0<r_{i}< a, \ i=1,\ldots N$ such that $K\subset \bigcup_{i=1}^{N}B_{x_{i}}(r_{i}/2)$ and $B_{x_{i}}(r_{i}) \subset U$ for all $i=1, \ldots N$.  Let $\{\zeta_{i}\}_{i=0}^{\infty}$ be a smooth partition of unity subordinate to the open sets $\{B_{x_{i}}(r_{i}/2)\}_{i=1}^{N}$, i.e., $0\leq \zeta_{i}\leq 1$, $\zeta_{i}\in C^{\infty}_{c}(B_{x_{i}}(r_{i}/2))$, $\sum_{i=1}^{N}\zeta_{i}=1$ on $K$.  With this, we know $\mathrm{spt}(\zeta_{i}\varphi)\subset B_{x_{i}}(r_{i}/2)$. Thus, the results above yield,
%\begin{equation}
%\begin{split}
%\langle f-Au,\varphi \rangle&=\langle f-Au, \sum_{i=1}^{N}\zeta_{i}\varphi \rangle =\sum_{i=1}^{N}\langle f-Au, \zeta_{i}\varphi\rangle \geq 0.
%\end{split}
%\end{equation}

\end{proof}

%\begin{rem}
%The proof of Proposition 2.5 in \cite{LM-2008} uses the fact that convolution commutes with the infinitesimal generator for a L\'{e}vy process.  Even though this does not hold true for our jump diffusion process $X$, we do not need this property since we know our value function $u$ is Lipschitz continuous and thus can show convergence in distribution directly using (\ref{ADistn}).  
%\end{rem}

\indent Upon knowing that $Au\leq f$ in $\mathcal{D}'(\mathbb{R}^{n})$ from Proposition \ref{MartingaleToDistn}, our next goal is to show that the distribution $Au$ is actually a function with $Au\in B_{2}(\mathbb{R}^{n})$.  This property not only describes the behavior of $Au$ at infinity but also would mean $Au\in L^{\infty}(\mathcal{O})$ for any bounded open set $\mathcal{O}$.  In turn, an application of Proposition \ref{AlmostLocal} would complete the regularity argument by allowing us to conclude $u\in W^{2,p}_{\mathrm{loc}}(\mathbb{R}^{n})$.  Below, we show $A(\mathcal{M}u)\geq -C(1+\abs{x}^{2})$ which combined with $Au\leq f$ in $\mathcal{D}'(\mathbb{R}^{n})$, $u\leq \mathcal{M}u$ in $\mathbb{R}^{n}$ and $Au=f$ in $\mathcal{D}'(\{u<\mathcal{M}u\})$ implies that $Au \in B_{2}(\mathbb{R}^{n})$.  
\subsubsection{Semi-concavity of $u$ and $\mathcal{M}u$}
The property $A(\mathcal{M}u)\geq -C(1+\abs{x}^{2})$ in $\mathcal{D}'(\mathbb{R}^{n})$ follows from the semi-concavity property of $u$ and $\mathcal{M}u$.  
\begin{defn}
\label{defnsemiconcave}
A continuous function $h$ from $\mathbb{R}^{n}$ to $\mathbb{R}^{n}$ is called \emph{semi-concave on} $\mathbb{R}^{n}$ if for every ball $B_{r}(0)$, $r>0$ there exists a constant $C_{r}>0$ such that $x\mapsto h(x)-C_{r}\abs{x}^{2}$ is concave on $B_{r}(0)$, i.e., for every $\abs{x}<r$, $\abs{y}<r$, we have
\begin{equation}
\begin{split}
\theta h(x)+(1-\theta)h(y)-h(\theta x+(1-\theta)y)\leq C_{r}\theta(1-\theta)\abs{x-y}^{2},
\end{split}
\end{equation}
for any $\theta\in [0,1]$.  If $h$ is continuous, this is equivalent to the condition
\begin{equation}
\begin{split}
h(x+z)-2h(x)+h(x-z)\leq C_{r}\abs{z}^{2},
\end{split}
\end{equation}
for all $z$ sufficiently small.  Equivalently, for any unit vector $\chi\in\mathbb{R}^{n}$ and constant $C>0$, we have
\begin{equation}
\frac{\partial^{2} h}{\partial \chi^{2}}\leq C, \ \text{in}\ \mathcal{D}'(\mathbb{R}^{n}).
\end{equation}  
\end{defn}      
\indent As observed in Section 4.2 in \cite{M-1987}  and Section 6 in \cite{DGW-2009}, in order to show the semi-concavity of $\mathcal{M}u$ on $\mathbb{R}^{n}$, it suffices to show the semi-concavity of $u$.  Indeed, for fixed $x\in \mathbb{R}^{n}$,
\begin{equation}
\begin{split}
\mathcal{M}u(x+z)-2\mathcal{M}u(x)+\mathcal{M}u(x-z)\leq u(y+z)-2u(y)+u(y-z),\\
\end{split}
\end{equation}
where $y:=x+\xi$ and $\xi\in \mathbb{R}^{n}$ is the limit of a convergent subsequence of a minimizing sequence $(\xi_{k})_{k=1}^{\infty}$ such that $u(x+\xi_{k})+B(\xi_{k})\rightarrow \mathcal{M}u(x)$.  The following lemma which, for instance, appears as Proposition 5.9 in Section 5.1.2 \cite{M-2008} assists in showing $u$ is semi-concave.
\begin{lem}
\label{estsemiconcave}
Let $X_{t}, X'_{t}, Z_{t}$ be three solutions of (\ref{sde}) for $t\geq 0$ with initial values $x,x',z$.  If $\alpha\geq \kappa$, as defined in (\ref{kappas}), then for $\psi_{\theta}(x,x',z):=\theta^2(1-\theta)^{2}\abs{x-x'}^{4}+\abs{\theta x+(1-\theta)x'-z}^{2}$ and under the assumptions (\ref{coeffAs}), and (\ref{coeffDerLip}), we have
\begin{equation}
\begin{split}
&\mathbb{E}\Bigg[ (\alpha-\kappa)\int_{0}^{t}\psi_{\theta}(X_{s}, X'_{s}, Z_{s})e^{-\alpha s}\mathrm{d}s+\psi_{\theta}(X_{t},X'_{t},Z_{t})e^{-\alpha t}\Bigg] \\
&\leq \psi_{\theta}(x,x',z), \ \text{for} \ t\geq 0.
\end{split}
\end{equation}
Moreover, there exists a constant $C>0$, depending on the bounds of $\sigma$, $j$ through (\ref{coeffAs}), and (\ref{coeffDerLip}), such that
\begin{equation}
\begin{split}
\mathbb{E}\Bigg[\underset{0\leq s \leq t}\sup \psi_{\theta}(X_{s},X'_{s},Z_{s})e^{-\alpha s}\Bigg]
\leq C\left(1+\frac{1}{\alpha-\kappa}\right)\psi_{\theta}(x,x',z), \ \text{for} \ t\geq 0.
\end{split}
\end{equation}
\end{lem}
\begin{proof}
The proof follows analogously to the proof of Lemma \ref{moment-est}.    Indeed, we consider $\psi_{\lambda,\theta}(x,x',z):=\lambda+\psi_{\theta}(x,x',z)$ and apply It\'{o}'s formula to find
\begin{equation}
\begin{split}
\mathrm{d}\psi_{\lambda,\theta}(X_{t},X'_{t},Z_{t})=a_{t}\mathrm{d}t+\sum_{k=1}^{d}b^{k}_{t}\mathrm{d}W^{k}_{t}+\int_{\mathbb{R}^{}}c(t,z)\tilde{N}(\mathrm{d}t,\mathrm{d}z),
\end{split}
\end{equation}
with $a_{t}\leq \kappa \psi_{\lambda,\theta}(X_{t},X'_{t},Z_{t})$.  As in Lemma \ref{moment-est}, we also have
\begin{equation}
\begin{split}
\sum_{k=1}^{d}\abs{b^{k}_{t}}^{2}+\int_{\mathbb{R}^{l}}\abs{c(t,z)}^{2}\nu(\mathrm{d}z)\leq C\abs{\psi_{\lambda,\theta}(X_{t},X'_{t},Z_{t})}^{2}, 
\end{split}
\end{equation}  
for some constant $C>0$.  Proceeding as in Lemma \ref{moment-est} completes the proof. \\
\end{proof}

We will apply this estimate as follows in Proposition \ref{semiconcave-prpn} below.  Let $Y_{t}(x)$ denote the solution of $(\ref{sde})$ with initial condition $Y_{0}(x)=x$.  From Lemma (\ref{estsemiconcave}), we have 
\begin{equation}
\begin{split}
\mathbb{E}\Bigg[&(\alpha-\kappa)\int_{0}^{t}\abs{\theta Y_{s}(x)+(1-\theta)Y_{s}(x')-Y_{s}(\theta x+(1-\theta)x')}^{2}e^{-\alpha s}\mathrm{d}s\\
&+\abs{\theta Y_{t}(x) + (1-\theta)Y_{t}(x')-Y_{t}(\theta x+(1-\theta)x')}^{2}e^{-\alpha t}\Bigg] \leq \theta^{2}(1-\theta)^{2}\abs{x-x'}^{4},  
\end{split}
\end{equation}
and 
\begin{equation}
\label{supest-semiconcave}
\begin{split}
\mathbb{E}\Bigg[&\underset{0\leq s \leq t}\sup\abs{\theta Y_{s}(x)+(1-\theta)Y_{s}(x')-Y_{s}(\theta x+(1-\theta)x')}^{2}e^{-\alpha s}\Bigg ]\\
&\leq C\left(1+\frac{1}{\alpha-\kappa}\right)\theta^{2}(1-\theta)^{2}\abs{x-x'}^{4}.  
\end{split}
\end{equation}
The following proposition asserts the semi-concavity property of $u$.
\begin{prpn}
\label{semiconcave-prpn}
Assume (\ref{coeffAs}), (\ref{coeffDerLip}), (\ref{fLip}), (\ref{fsemiconcave}), and suppose $r$ is sufficiently large.  Then $u$ is semi-concave on $\mathbb{R}^{n}$.  
\end{prpn}

\begin{proof}
Fix an admissible control $V$.  The value function $u(x)$ will be semi-concave if $J_{x}[V]$ is semi-concave since the infimum of semi-concave functions is semi-concave.  Appealing to Definition \ref{defnsemiconcave}, we show
 \begin{equation}
 \begin{split}
 \theta J_{x}[V]+(1-\theta)J_{x'}[V]-J_{\theta x + (1-\theta)x'}[V] \leq C \theta(1-\theta)\abs{x-x'}^{2}
 \end{split}
 \end{equation}
Define
\begin{equation}
\label{kappas}
\begin{split}
\kappa&:= \underset{x,x',y,\theta} \sup\{2\kappa_{\tilde{b}}+\kappa_{\sigma}+\kappa_{j}\}, \ \text{with} \\
\end{split}
\end{equation}

\begin{equation*}
\begin{split}
\kappa_{\tilde{b}}&:=\sum_{i}2\theta^{2}(1-\theta)^{2}\abs{x-x'}^{2}(x_{i}-x'_{i})[\tilde{b}_{i}(x)-\tilde{b}_{i}(x')]\\
&+\sum_{i}(\theta x_{i}+(1-\theta)x_{i}'-y_{i})[\theta \tilde{b}_{i}(x)+(1+\theta)\tilde{b}_{i}(x')-\tilde{b}_{i}(y)], 
\end{split}
\end{equation*}

\begin{equation*}
\begin{split}
\kappa_{\sigma}&:=\theta^{2}(1-\theta)^{2}\Bigg [\sum_{h,k}2\abs{x-x'}^{2}+4(x_{h}-x'_{h})^{2}(\sigma_{hk}(x)-\sigma_{hk}(x'))^{2}\\
&+\sum_{i\neq j, k}4(x_{i}-x_{i}')(x_{j}-x_{j}')(\sigma_{ik}(x)-\sigma_{ik}(x'))(\sigma_{jk}(x)-\sigma_{jk}(x'))\Bigg]\\
&+\sum_{i,k}[\theta\sigma_{ik}(x)+(1-\theta)\sigma_{ik}(x')-\sigma_{ik}(y)]^{2}, 
\end{split}
\end{equation*}

\begin{equation*}
\begin{split}
\kappa_{j}&:=\int_{\mathbb{R}^{l}}\Big[\abs{x-x'+j(x,z)-j(x',z)}^{4}-\abs{x-x'}^{4}\\
& -\sum_{i}4\abs{x-x'}(x_{i}-x_{i}')\times (j_{i}(x,z)-j_{i}(x',z)\Big]\nu(\mathrm{d}z)\\
&+\int_{\mathbb{R}^{l}}\Big[\abs{\theta x+(1-\theta)x'-y+(\theta j(x,z)+(1-\theta)j(x',z)-j(y,z))}^{2}\\
&-\abs{\theta x+(1-\theta)x'-y}^{2}\\
&-\sum_{i}2(\theta x_{i}+(1-\theta)x_{i}'-y_{i})\times(\theta j_{i}(x,z)+(1-\theta)j_{i}(x',z)-j_{i}(y,z))\Big]\nu(\mathrm{d}z),
\end{split}
\end{equation*}
where $x,x',z\in \mathbb{R}^{n}$, $\theta\in [0,1]$ and $\beta\leq \kappa<\infty$ due to (\ref{coeffAs}), (\ref{coeffDerLip}) (see Section 5.2.1 in \cite{M-2008} for a similar discussion).  We have for $\alpha\geq \kappa\geq \beta$,
\begin{equation}
\begin{split}
&\theta J_{x}[V]+(1-\theta)J_{x'}[V]-J_{\theta x + (1-\theta)x'}[V]\\
&=\mathbb{E}\Bigg[\int_{0}^{\infty}[\theta f(Y_{t}(x))+(1-\theta)f(Y_{t}(x'))-f(Y_{t}(\theta x +(1-\theta)x'))]e^{-r t}\mathrm{d}t\Bigg]\\
&=\mathbb{E}\Bigg[\int_{0}^{\infty}[\theta f(Y_{t}(x))+(1-\theta)f(Y_{t}(x'))-f(\theta Y_{t}(x)+(1-\theta)Y_{t}(x'))\\
& \quad+f(\theta Y_{t}(x)+(1-\theta)Y_{t}(x'))-f(Y_{t}(\theta x +(1-\theta)x'))]e^{-r t}\mathrm{d}t\Bigg]\\
&\leq C\theta(1-\theta)\int_{0}^{\infty}e^{-r t}\mathbb{E}[\abs{Y_{t}(x)-Y_{t}(x')}^{2}]\mathrm{d}t\\
&\quad +C_{f}\int_{0}^{\infty}e^{-r t}\mathbb{E}[\abs{\theta Y_{t}(x)+(1-\theta)Y_{t}(x')-Y_{t}(\theta x+(1-\theta)x')}]\mathrm{d}t\\
&\leq C\theta(1-\theta)\abs{x-x'}^{2}\int_{0}^{\infty}e^{-(r-\alpha)t}\mathrm{d}t\\
& \quad C_{f}C^{1/2}\left(1+\frac{1}{\alpha-\kappa}\right)^{1/2}\theta(1-\theta)\abs{x-x'}^{2}\int_{0}^{\infty}e^{-(r-\alpha) t}\mathrm{d}t\\
&\leq C\theta(1-\theta)\abs{x-x'}^{2}.
\end{split}
\end{equation}
The first inequality follows using semi-concavity and Lipschitz continuity of $f$.  The second inequality follows using a standard estimate for the difference of solutions for (\ref{sde}) (c.f. Theorem 5.6 in \cite{M-2008}) and (\ref{supest-semiconcave}).    
\end{proof}
\subsubsection{$u\in W^{2,p}_{\mathrm{loc}}(\mathbb{R}^{n})$}
\label{W2plocRn} Using the semi-concavity property of $\mathcal{M}u$ on $\mathbb{R}^{n}$, the following mollification argument shows that $A(\mathcal{M}u)\geq -C(1+\abs{x}^{2})$ in $\mathcal{D}'(\mathbb{R}^{n})$ for some constant $C>0$.  With $A:=(-\mathcal{L}_{D}-I+r)$ as in (\ref{Aoperator}),
Since $\mathcal{M}u$ is semi-concave on $\mathbb{R}^{n}$, we know
\begin{equation}
\begin{split}
\mathcal{M}u(x+\rho\chi)+\mathcal{M}u(x-\rho\chi)-2\mathcal{M}u(x)\leq K\rho^{2}, \ x\in \mathbb{R}^{n},
\end{split}
\end{equation}
for any $\rho>0$ and unit vector $\chi\in \mathbb{R}^{n}$ and non-negative constant $K$.  Below, $C$ denotes a generic constant independent of $\varepsilon$.  Let $g=\mathcal{M}u$ and denote $g^{\varepsilon}$ its mollification on $\mathbb{R}^{n}$.  We first show that $A(g^{\varepsilon}(x))\geq -C(1+\abs{x}^{2})$ for $C$ independent of $\varepsilon$.  We proceed by estimating each term in $A(g^{\varepsilon})$.  For $x\in \mathbb{R}^{n}$, $\rho>0$ and unit vector $\chi\in \mathbb{R}^{n}$, 
\begin{equation}
\begin{split}
&\frac{1}{\rho^{2}}\left(g^{\varepsilon}(x+\rho\chi)+g^{\varepsilon}(x-\rho\chi)-2g^{\varepsilon}(x)\right)\\
& =\frac{1}{\rho^{2}}\int_{B_{\varepsilon}(0)}\left(g(x-z+\rho\chi)+g(x-z-\rho\chi)-2g(x-z)\right)\eta^{\varepsilon}(z)\mathrm{d}z\\
& \leq K\int_{B_{\varepsilon}(0)}\eta^{\varepsilon}(z)\mathrm{d}z.
\end{split}
\end{equation}
Sending $\rho\rightarrow 0$, yields $\chi^{T}\nabla^{2}g^{\varepsilon}(x)\chi\leq K$.  Using this, we have
\begin{equation}
\begin{split}
Tr[\sigma(x)\sigma(x)^{T}\nabla^{2}g^{\varepsilon}(x)]&=\sum_{i=1}^{n}\sigma_{i}^{T}(x)\nabla^{2}g^{\varepsilon}(x)\sigma_{i}(x)\\
&\leq K\sum _{i,j=1}^{n}\abs{\sigma_{ij}(x)}^{2}\\
&\leq C(1+\abs{x}^{2}).
\end{split}
\end{equation}
Using Lipschitz continuity of $\tilde{b},g$, we know 
\begin{equation}
\begin{split}
\abs{\tilde{b}(x)\cdot \nabla g^{\varepsilon}(x)}\leq \abs{\tilde{b}(x)}\abs{\nabla g^{\varepsilon}(x)} \leq C(1+\abs{x})nC_{\mathcal{M}u}&=C(n)(1+\abs{x}),\\
&\leq C(n)(2+\abs{x}^{2})\\
&\leq C(1+\abs{x}^{2}),
\end{split}
\end{equation}
where $C_{\mathcal{M}u}$ is the Lipschitz constant for $\mathcal{M}u$, and $C(n)$ is a constant depending on the dimension $n$.  Next,
\begin{equation}
\begin{split}
\abs{g^{\varepsilon}(x)-g(x)}&\leq \int_{B_{\varepsilon}(0)}\abs{g(x-z)-g(x)}\eta^{\varepsilon}(z)\mathrm{d}z\\
&\leq C_{\mathcal{M}u}\int_{B_{\varepsilon}(0)}\abs{z}\eta^{\varepsilon}(z)\mathrm{d}z\\
&\leq \varepsilon C_{\mathcal{M}u}.
\end{split}
\end{equation}
Then, for all $\varepsilon\in \left(0,\frac{1}{C_{\mathcal{M}u}}\right)$, we have 
\begin{equation}
\begin{split}
\abs{g^{\varepsilon}(x)}\leq \abs{g(x)}+1 \leq C(1+\abs{x})\leq C(1+\abs{x}^{2}).
\end{split}
\end{equation}
With regard to the integro term, we have
\begin{equation}
\begin{split}
&\abs{\int_{\mathbb{R}^{n}}[g^{\varepsilon}(x+j(x,z))-g^{\varepsilon}(x)-\sum_{i=1}^{n}j_{i}(x,z)\partial_{x_{i}}g^{\varepsilon}(x)]\nu(\mathrm{d}z)}\\
&\leq \int_{\mathbb{R}^{n}}\left(\int_{0}^{1}(1-\theta)\abs{j(x,z)^{T}\cdot \nabla^{2}g^{\varepsilon}(x+\theta j(x,z))\cdot j(x,z)}\mathrm{d}\theta\right)\nu(\mathrm{d}z)\\
&\leq \int_{\mathbb{R}^{n}}\frac{K}{2}\abs{j(x,z)}^{2}\nu(\mathrm{d}z)\\
&\leq C(1+\abs{x}^{2}),
\end{split}
\end{equation}  
Gathering these estimates, we have for all $\varepsilon\in \left(0,\frac{1}{C_{\mathcal{M}u}}\right)$,
\begin{equation}
\begin{split}
A(g^{\varepsilon}(x))&=-\frac{1}{2}Tr[\sigma(x)\sigma(x)^{T}\nabla^{2}g^{\varepsilon}(x)]-\tilde{b}(x)\cdot\nabla g^{\varepsilon}(x)+rg^{\epsilon}(x)\\
& \quad -\int_{\mathbb{R}^{n}}[g^{\varepsilon}(x+j(x,z))-g^{\varepsilon}(x)-\sum_{i=1}^{n}j_{i}(x,z)\partial_{x_{i}}g^{\varepsilon}(x)]\nu(\mathrm{d}z)
\\
&\geq -C(1+\abs{x}^{2}),
\end{split}
\end{equation}
where $C$ depends upon the dimension $n$ but is independent of $\varepsilon$.  Now, this pointwise estimate implies that $A(g^{\varepsilon})\geq -C(1+\abs{x}^{2})$ in $\mathcal{D}'(\mathbb{R}^{n})$.  Since $g^{\varepsilon}\rightarrow g$ in $L^{1}_{\mathrm{loc}}(\mathbb{R}^{n})$ and $g_{x_{i}}^{\varepsilon}\rightarrow g_{x_{i}}$ in $L^{1}_{\mathrm{loc}}(\mathbb{R}^{n})$ (recall, $g$ is Lipschitz continuous), we know from (\ref{ADistn}) that $\langle Ag^{\varepsilon},\varphi \rangle\rightarrow \langle Ag,\varphi \rangle$ for every $\varphi\in \mathcal{D}(\mathbb{R}^{n})$.  Thus, $A(\mathcal{M}u)\geq -C(1+\abs{x}^{2})$ in $\mathcal{D}'(\mathbb{R}^{n})$.\\
\indent At this point, we know 
\begin{equation}
\begin{split}
\begin{cases}
-C(1+\abs{x}^{2})\leq Au \leq f, \ \text{in} \ \mathcal{D}'(\{u=\mathcal{M}u\}),\\
Au=f, \ \text{in} \ \mathcal{D}'(\{u<\mathcal{M}u\}).\\
\end{cases}
\end{split}
\end{equation}
From the above inequality, one can easily conclude that $Au$ exists as a function on $\{u=\mathcal{M}u\}$.  One way to see this is to note that 
\begin{equation}
\label{AuAsFcn}
\begin{split}
\int_{\mathcal{O}}[f+C(1+\abs{x}^2)]\varphi \ \mathrm{d}x =\int_{\mathcal{O}}\varphi \ \mathrm{d}(\mu_{1}+\mu_{2}), \ \varphi\in \mathcal{D}(\mathcal{O}),
\end{split}
\end{equation}
for any bounded open set $\mathcal{O}\subset \{u=\mathcal{M}u\}$ and where $\mu_{1},\mu_{2}$ are measures corresponding to the positive distributions $f-Au$ and $Au+C(1+\abs{x}^{2})$ respectively. Since $\mu_{1}+\mu_{2}$ is a positive measure corresponding to a function, it is absolutely continuous with respect to the Lebesgue measure, i.e. $\mu_{1}+\mu_{2} \ll \ell$ (Lebesgue measure) on $\mathcal{O}$ which then implies $\mu_{1}, \mu_{2}\ll\ell$ on $\mathcal{O}$.  Now, by definition of $\mu_{1}$ and $\mu_{2}$, we observe that $Au$ is a function.  Hence, $Au$ exists as a function and satisfies $\abs{Au(x)}\leq C(1+\abs{x}^{2})$, i.e., $Au(x)\in B_{2}(\mathbb{R}^{n})$.  Knowing $Au(x)\in B_{2}(\mathbb{R}^{n})$ allows us to apply Proposition \ref{AlmostLocal} with $f=Au$ over any bounded open set $\mathcal{O}$.  Thus, we have $u\in W^{2,p}_{\mathrm{loc}}(\mathbb{R}^{n})$ for $p\in(1,\infty)$ as desired.

\section*{Acknowledgments}
E. Bayraktar is supported in part by the National Science Foundation under an applied mathematics research grant and a Career grant, DMS-0906257 and DMS-0955463, respectively, and in part by the Susan M. Smith Professorship. We also would like to thank Christopher Link for his REU work related to this subject.

\appendix

\renewcommand{\theequation}{A.\arabic{equation}}
\renewcommand{\thetheorem}{A.\arabic{theorem}}
\renewcommand{\thedefinition}{A.\arabic{definition}}
\renewcommand{\thelemma}{A.\arabic{lemma}}
\section{Proofs of some technical results}

\subsubsection*{Proof of Lemma~\ref{u*lip}}

Given an admissible control $V$ and two initial states $x_{1},x_{2}$, denote by $X^{i}_{t}$ the solution of (\ref{sde}).  Set $Y_{t}=X^{1}_{t}-X^{2}_{t}$ and apply It\^{o}'s formula with $\varphi(y,t)=\abs{y}^{2}e^{-\alpha t}$ to obtain
\begin{equation}\label{eq:dY}
\begin{split}
\mathrm{d}\varphi(Y_{t},t)=a_{t}\mathrm{d}t+\sum_{k=1}^{d}b^{k}_{t} \ \mathrm{d}W^{k}_{t}+\int_{\mathbb{R}^{l}}c(t,z)\tilde{N}(\mathrm{d}t,\mathrm{d}z), \ \text{where}
\end{split}
\end{equation}
\begin{equation*}
\begin{split}
a_{t}&:=\partial_{t}\varphi(Y_{t},t)+\sum_{i=1}^{n}[\tilde{b}_{i}(X^{1}_{t})-\tilde{b}_{i}(X^{2}_{t})]\partial_{i}\varphi(Y_{t},t)\\
&+\frac{1}{2}\sum_{i,j=1}^{n}\Bigg(\sum_{k=1}^{d}[\sigma_{ik}(X^{1}_{t})-\sigma_{ik}(X^{2}_{t})][\sigma_{jk}(X^{1}_{t})-\sigma_{jk}(X^{2}_{t})]\Bigg)\partial^{2}_{ij}\varphi(Y_{t},t)\\
&+\int_{\mathbb{R}^{l}}[\varphi(Y_{t}+j(X^{1}_{t},z)-j(X^{2}_{t},z),t)-\varphi(Y_{t},t)\\
&-\sum_{i=1}^{n}[j_{i}(X^{1}_{t},z)-j_{i}(X^{2}_{t},z)]\partial_{i}\varphi(Y_{t},t)\nu(\mathrm{d}z),\\
b^{k}_{t}&:=\sum_{i=1}^{n}(\sigma_{ik}(X^{1}_{t})-\sigma_{ik}(X^{2}_{t}))\partial_{i}\varphi(Y_{t},t),\\
c(t,z)&:= \varphi(Y_{t}+j(X^{1}_{t},z)-j^{\epsilon}(X^{2}_{t},z),t)-\varphi(Y_{t},t),
\end{split}
\end{equation*}
and, 
\begin{equation}\label{eq:phider}
\begin{split}
\partial_{t}\varphi(y,t)&=-\alpha \varphi(y,t),
\;\partial_{i}\varphi(y,t)=2y_{i}\abs{y}^{-2}\varphi(y,t)=2y_{i}e^{-\alpha t},
\;\partial^{2}_{ij}\varphi(y,t)=2\delta_{ij}\abs{y}^{-2}\varphi(y,t),
\end{split}
\end{equation}
where $\delta_{ij}=1$ if $i=j$ and $0$ otherwise.  Define
\begin{equation}
\label{betas}
\begin{split}
\beta&:= \underset{x,x'\in \mathbb{R}^{n}} \sup\{2\beta_{\tilde{b}}+\beta_{\sigma}+\beta_{j}\}, \ \text{with} \\
\beta_{\tilde{b}}&:=\sum_{i}\frac{(x_{i}-x'_{i})[\tilde{b}_{i}(x)-\tilde{b}_{i}(x')]}{\abs{x-x'}^{2}}, \\
\beta_{\sigma}&:=\sum_{i,k}\frac{[\sigma_{ik}(x)-\sigma_{ik}(x')]^{2}}{\abs{x-x'}^{2}},\\ 
\beta_{j}&:=\int_{\mathbb{R}^{l}}\Big [ \abs{x-x'+j(x,z)-j(x',z)}^{2}-\abs{x-x'}^{2}\\
& \quad \quad -\sum_{i}2(x_{i}-x'_{i})[j_{i}(x,z)-j_{i}(x',z)]\Big]\abs{x-x'}^{-2}\nu(\mathrm{d}z),
\end{split}
\end{equation}
where $\beta<\infty$ due to (\ref{coeffAs}). Using (\ref{betas}) and taking $\alpha>\beta$, we find
\begin{equation*}
\begin{split}
\mathbb{E}[\varphi(Y_{t},t)]-(x_{1}-x_{2})^{2}&\leq (-\alpha+\beta)\int_{0}^{t}\mathbb{E}[\varphi(Y_{s},s)]\mathrm{d}s,
\end{split}
\end{equation*}
which implies $\mathbb{E}[\abs{X^{1}_{t}-X^{2}_{t}}]\leq e^{\beta t/2}\abs{x_{1}-x_{2}}$ by Gronwall's and Jensen's inequality.  Using (\ref{fLip}) and since $r$ is sufficiently large, we have $J_{x_{1}}[V]-J_{x_{2}}[V]\leq C_{u}\abs{x_{1}-x_{2}}$ with $C_{u}=C_{f}/(r-\beta/2)$.  Subsequently,
\begin{equation*}
\begin{split}
u(x_{1})\leq J_{x_{1}}[V]\leq J_{x_{2}}[V]+C_{u}\abs{x_{1}-x_{2}}.
\end{split}
\end{equation*}
Taking the infimum over all admissible controls with initial state $x_{2}$ yields the desired inequality.  Now, exchanging the roles of $x_{1}, x_{2}$ completes the proof. 

\subsubsection*{Proof of Lemma~\ref{IntSobEst}}

Let $\eta\in(0,1]$ be determined later.  Based on (\ref{BdsBigjumps}), we know
\begin{equation}
\begin{split}
\eta^{\gamma-1}\int_{\{\eta\leq j_{0}(z)<1\}}j_{0}(z)\nu(\mathrm{d}z)\leq \int_{\{j_{0}(z)<1\}}[j_{0}(z)]^{\gamma}\nu(\mathrm{d}z)\leq C_{0}, 
\end{split}
\end{equation}

\begin{equation}
\label{epsilonLpEst}
\begin{split}
\int_{\{j_{0}(z)<\eta\}}[j_{0}(z)]^{2}\nu(\mathrm{d}z)\leq \eta^{2-\gamma}r(\eta),
\end{split}
\end{equation}
where the module of integrability is given by 
\begin{equation}
\begin{split}
r(\eta)=\int_{\{j_{0}(z)<\eta\}}[j_{0}(z)]^{\gamma}\nu(\mathrm{d}z).
\end{split}
\end{equation}
Now, we write $I\varphi=I^{1}_{\eta}\varphi + I^{2}_{\eta}\varphi + I^{3}_{\eta}\varphi$ with 
\begin{equation}
\begin{split}
I^{1}_{\eta}\varphi&=\int_{\{j_{0}(z)\geq 1\}}\varphi(\cdot+j(\cdot,z))-\varphi(\cdot) \nu(\mathrm{d}z),\\
I^{2}_{\eta}\varphi &=\int_{\{\eta\leq j_{0}(z)<1\}}\varphi(\cdot+j(\cdot,z))-\varphi(\cdot)-\nabla\varphi(\cdot)\cdot j(\cdot,z)\nu(\mathrm{d}z),\\
I^{3}_{\eta}\varphi &=\int_{0}^{1}(1-\theta)\mathrm{d}\theta \int_{\{j_{0}(z)<\eta\}}j(\cdot,z)\cdot \nabla^{2}\varphi(\cdot+\theta j(\cdot,z))\cdot j(\cdot,z)\nu(\mathrm{d}z).
\end{split}
\end{equation}
Using Lipschitz continuity, we have $\abs{I^{1}_{\eta}\varphi}\leq C_{\varphi}\int_{\{j_{0}(z)\geq 1\}}j_{0}(z) \nu(\mathrm{d}z)\leq C_{\varphi}C_{0}$ and $\abs{I^{2}_{\eta}\varphi}\leq 2 C_{\varphi}\int_{\{\eta\leq j_{0}(z)<1\}}j_{0}(z)\nu(\mathrm{d}z)\leq 2C_{\varphi}C_{0}\eta^{1-\gamma}$.  For the last term, we have
\begin{equation}
\begin{split}
\abs{I^{3}_{\eta}\varphi}\leq \int_{0}^{1}\mathrm{d}\theta \int_{\{j_{0}(z)<\eta\}}\abs{j_{0}(z)}^{2}\abs{\nabla^{2}\varphi(\cdot+\theta j(\cdot,z))}\nu(\mathrm{d}z).
\end{split}
\end{equation}
Using this, we can estimate the $L^{p}$ norm as follows
\begin{equation*}
\begin{split}
\norm{I^{3}_{\eta}\varphi}_{L^{p}(\mathcal{O})}^{p}&\leq \int_{\mathcal{O}}\mathrm{d}x\int_{0}^{1}\mathrm{d}\theta\left(\int_{\{j_{0}(z)\leq\eta\}}\abs{j_{0}(z)}^{2}\abs{\nabla^{2}\varphi(x+\theta j(x,z))}\right)^{p} \nu(\mathrm{d}z)\\
& \leq \int_{\mathcal{O}}\mathrm{d}x\int_{0}^{1}\mathrm{d}\theta\left(\int_{\{j_{0}(z)\leq \eta\}}\abs{j_{0}(z)}^{2}\nu(\mathrm{d}z) \right)^{\frac{p}{q}}\\
&\quad \times \left(\int_{\{j_{0}(z)\leq \eta\}} \abs{j_{0}(z)}^{2}\abs{\nabla^{2}\varphi(x+\theta j(x,z))}^{p}\nu(\mathrm{d}z)\right)\\
&\leq (\eta^{2-\gamma}r(\eta))^{p}\norm{\nabla^{2}\varphi}^{p}_{L^{p}(\mathcal{O}^{\eta})},
\end{split}
\end{equation*}
Above, we use Fubini's theorem, Jensen's inequality, and the H\"{o}lder inequality with $1/p+1/q=1$.  Thus, $\norm{I^{3}_{\eta}\varphi}_{L^{p}(\mathcal{O})}\leq \eta^{2-\eta}r(\eta)\norm{\nabla^{2}\varphi}_{W^{2,p}(\mathcal{O}^{\eta})}$.  
From the above estimates, we find
\begin{equation}
\begin{split}
\norm{I\varphi}_{L^{p}(\mathcal{O})}\leq \eta^{2-\gamma}r(\eta)\norm{\nabla^{2}\varphi}_{L^{p}(\mathcal{O}^{\eta})}+C_{0}(1+2\eta^{1-\gamma})C_{\varphi}.
\end{split}
\end{equation}
Note that the module of integrability satisfies $r(\eta) \rightarrow 0$ as $\eta \rightarrow 0$.  Now choose $\eta$ small enough so that $\eta^{2-\gamma}r(\eta)<\varepsilon$ and $\eta<\varepsilon$.

\subsubsection*{Proof of Proposition~\ref{AlmostLocal}}
Let $C$ denote a generic constant throughout this proof.  Let $R\in (0,\mathrm{dist}(\mathcal{O}',\partial \mathcal{O}))$.  Consider $B_{R}(x_{0})$ (or simply $B_{R}$) for $x_{0}\in \mathcal{O}'$.  For a constant $0<\delta<1$ to be determined later, consider a smooth cut-off function $\zeta^{\delta}$ satisfying
\begin{equation}
\begin{split}
\begin{cases}
\zeta^{\delta}\equiv 1\ \text{on} \ B_{\frac{\delta}{2}R}, \zeta^{\delta}\equiv 0 \ \text{on} \ \mathbb{R}^{n}\setminus B_{\frac{3\delta}{4}R},\\
0\leq \zeta^{\delta} \leq 1. 
\end{cases}
\end{split}
\end{equation}
Moreover, $\zeta^{\delta}$ can be chosen to satisfy $\abs{\partial_{i}\zeta^{\delta}}\leq \frac{C}{\delta}$, $\abs{\partial^{2}_{ij}\zeta^{\delta}}\leq \frac{C}{\delta^{2}}$ for a constant $C$.
The function $w:=\zeta^{\delta} v$ satisfies
\begin{equation}
\begin{split}
\begin{cases}
(-\mathcal{L}_{D}+r)w=\zeta^{\delta} Iv(x)+\zeta^{\delta} f(x)+h(x) & x\in B_{\frac{3\delta}{4}R},\\
w(x)=0 & x\in \partial B_{\frac{3\delta}{4}R}, 
\end{cases}
\end{split}
\end{equation}
where $h(x):=-\sum_{i,j=1}^{n}a_{ij}(\partial^{2}_{ij}\zeta^{\delta}\cdot v+2 \partial_{i}\zeta^{\delta}\cdot \partial_{j}v)-\sum_{i=1}^{n}b_{i}\cdot\partial_{i}\zeta^{\delta} \cdot v$.  For this classical Dirichlet problem, there exists a constant $C$ independent of $w$ such that
\begin{equation}
\label{mainest}
\norm{w}_{W^{2,p}(B_{\frac{3\delta}{4}R})}\leq C\left(\norm{\zeta^{\delta} Iv}_{L^{p}(B_{\frac{3\delta}{4}R})}+\norm{\zeta^{\delta} f}_{L^{p}(B_{\frac{3\delta}{4}R})}+\norm{h}_{L^{p}(B_{\frac{3\delta}{4}R})}\right).
\end{equation}
We now estimate the terms on the right-hand side of (\ref{mainest}) individually.  For the first term, 
\begin{equation}
\begin{split}
\norm{\zeta^{\delta} Iv}_{L^{p}(B_{\frac{3\delta}{4}R})}\leq \norm{Iv}_{L^{p}(B_{\frac{3\delta}{4}R})}\leq \frac{\delta}{4}\norm{v}_{W^{2,p}(B_{\delta R})}+C\left(\frac{\delta}{4}\right)C_{v},
\end{split}
\end{equation}
where the first inequality follows from the choice of $\zeta^{\delta}$; the second inequality follows from Lemma \ref{IntSobEst} with $\varepsilon=\frac{\delta}{4}$.  Next, it is clear that $\norm{\zeta^{\delta} f}_{L^{p}(B_{\frac{3\delta}{4}R})}\leq \norm{f}_{L^{p}(B_{\frac{3\delta}{4}R})}$.  Now, we will estimate $\norm{h}_{L^{p}(B_{\frac{3\delta}{4}R})}$.  It follows from our choice of $\zeta^{\delta}$ that
\begin{equation*}
\begin{split}
\norm{\sum_{i,j=1}^{n}a_{ij}\partial^{2}_{ij}\zeta^{\delta}\cdot v}_{L^{p}(B_{\frac{3\delta}{4}R})}&\leq C\cdot \norm{v}_{L^{\infty}(B_{\frac{3\delta}{4}R})}\cdot \norm{\partial^{2}_{ij}\zeta^{\delta}}_{L^{p}(B_{\frac{3\delta}{4}R}\setminus B_{\frac{\delta}{2}R})}\\
& \leq C\cdot \norm{v}_{L^{\infty}(B_{\frac{3\delta}{4}R})}\cdot\delta^{\frac{n-2p}{p}}, \ \text{and}, \\
\norm{\sum_{i,j=1}^{n}2a_{ij}\partial_{i}\zeta^{\delta}\cdot \partial_{j}v}_{L^{p}(B_{\frac{3\delta}{4}R})}&\leq C\cdot C_{v}\cdot \delta^{\frac{n-p}{p}}, \\
\norm{\sum_{i=1}^{n}b_{i}\cdot \partial_{i}\zeta^{\delta}\cdot v}_{L^{p}(B_{\frac{3\delta}{4}R})}
& \leq C\cdot \norm{v}_{L^{\infty}(B_{\frac{3\delta}{4}R})}\cdot \delta^{\frac{n-p}{p}}. 
\end{split}
\end{equation*}
Using the above estimates, we obtain
\begin{equation}
\begin{split}
&\norm{v}_{W^{2,p}(B_{\frac{\delta}{2}R})}\leq \norm{w}_{W^{2,p}(B_{\frac{3\delta}{4}R})}\leq C\frac{\delta}{4}\norm{v}_{W^{2,p}(B_{\delta R})}\\
& \quad +C\left(\norm{v}_{L^{\infty}(B_{\frac{3\delta}{4}R})}+C_{v}\right)(1+\delta^{\frac{n-p}{p}}+\delta^{\frac{n-2p}{p}})+ C\norm{f}_{L^{p}(B_{\frac{3\delta}{4}R})}.
\end{split}
\end{equation}
%In the above estimate, note that \eta^{1-\beta}=(0.5)^{1-\beta}*max((\frac{\delta}{4})^{1-\beta},C_{1}^{(\beta-1)^{2}}*(\delta/4)(1-\beta)(\beta-1)).  Since \beta\in [0,1), we can, for large enough constant C, get rid of the max function (i.e. only need to look at right hand term).  
Multiplying $\delta^{2}$ on both sides of the previous inequality produces
\begin{equation}
\begin{split}
\delta^{2}\norm{v}_{W^{2,p}(B_{\frac{\delta}{2}R})}\leq C\delta\left(\frac{\delta}{2}\right)^{2}\norm{v}_{W^{2,p}(B_{\delta R})}+K(\delta),
\end{split}
\end{equation}
where $K(\delta):=C\cdot \left(\norm{v}_{L^{\infty}(B_{\frac{3\delta}{4}R})}+C_{v}\right)\cdot (\delta^{2}+\delta^{\frac{n+p}{p}}+\delta^{\frac{n}{p}})+ \norm{f}_{L^{p}(B_{\frac{3\delta}{4}R})}$.
Denote $F(\tau):=\tau^{2}\norm{v}_{W^{2,p}(B_{\frac{\delta}{2}R+(\delta-\tau)})}$.
The previous inequality yields the following recursive inequality
$F(\delta)\leq C\delta \  F\left(\frac{\delta}{2}\right)+K(\delta)$.  Choosing $0<\delta<1$ such that $\delta \leq \frac{1}{2C}$, we obtain 
$F(\delta)\leq \frac{1}{2}F\left(\frac{\delta}{2}\right)+K(\delta)$.  Now iterating the recursive inequality and noting that $K(\delta)$ is an increasing function, we obtain
\begin{equation}
\begin{split}
F(\delta)\leq \sum_{i=0}^{\infty}\frac{1}{2^{i}}K\left(\frac{\delta}{2^{i}}\right)\leq \sum_{i=0}^{\infty}\frac{1}{2^{i}}K(\delta)=2K(\delta).
\end{split}
\end{equation}
Hence,
\begin{equation}
\begin{split}
\norm{v}_{W^{2,p}(B_{\frac{\delta}{2}R})}&\leq 2\left(C\left(\norm{v}_{L^{\infty}(B_{\frac{3\delta}{4}R})}+C_{v}\right)(\delta^{2}+\delta^{\frac{1+p}{p}}+\delta^{\frac{1}{p}})+ \norm{f}_{L^{p}(B_{\frac{3\delta}{4}R})}\right),\\
&\leq C\left(\norm{f}_{L^{p}(B_{\frac{3\delta}{4}R})}+C_{v}+\norm{v}_{L^{\infty}(B_{\frac{3\delta}{4}R})}\right).
\end{split}
\end{equation}
If we cover $\mathcal{O}'$ with a finite number of balls of radius $\frac{\delta}{2}R$, then the estimate of the proposition follows.  

\subsubsection*{Proof of Lemma~\ref{HolderReg}}
Let $C$ denote a generic constant unless specified otherwise.  First, we estimate $\sup_{\overline{\Omega}}\abs{I\varphi}$.  For any $x\in \overline{\Omega}$,
\begin{equation}
\begin{split}
\abs{I\varphi(x)}&\leq \int_{\{j_{0}(z)< 1\}}\abs{\varphi(x+j(x,z))-\varphi(x)-\nabla\varphi(x)\cdot j(x,z)}\nu(\mathrm{d}z)\\
& \quad +\int_{\{j_{0}(z)\geq1\}}\abs{\varphi(x+j(x,z))-\varphi(x)}\nu(\mathrm{d}z)\\
&\leq \int_{\{j_{0}(z)< 1\}}\int_{0}^{1}\abs{\nabla\varphi(x+\theta j(x,z))\cdot j(x,z)-\nabla\varphi(x)\cdot j(x,z)}\mathrm{d}\theta \ \nu(\mathrm{d}z)\\
& \quad +C_{\varphi}\int_{\{j_{0}(z)\geq 1\}}j_{0}(z)\nu(\mathrm{d}z)\\
&\leq \norm{\varphi}_{C^{1,\alpha}(\overline{\Omega^{1}})}\int_{\{j_{0}(z)< 1\}}[j_{0}(z)]^{1+\alpha}\nu(\mathrm{d}z)+C_{\varphi}\int_{\{j_{0}(z)\geq 1\}}j_{0}(z)\nu(\mathrm{d}z)\\
&\leq C_{0}\left(C_{\varphi}+ \norm{\varphi}_{C^{1,\alpha}\left(\overline{\Omega^{1}}\right)}\right).
\end{split}
\end{equation}
Next, we show $I\varphi$ is H\"{o}lder continuous.  Let $x_{1},x_{2}\in \overline{\Omega}$ and set $\delta=\abs{x_{1}-x_{2}}^{\frac{1}{2}}\wedge 1$.  Consider $\abs{I\varphi(x_{1})-I\varphi(x_{2})}\leq I_{1}+I_{2}+I_{3}$ in which
\[
\begin{split}
I_{1}&:=\int_{\{j_{0}(z)\leq \delta\}}(\abs{\varphi(x_{1}+j(x_{1},z))-\varphi(x_{1})-j(x_{1},z)\cdot \nabla\varphi(x_{1})}\\
& \quad  +\abs{\varphi(x_{2}+j(x_{2},z))-\varphi(x_{2})-j(x_{2},z)\cdot \nabla\varphi(x_{2})}) \ \nu(\mathrm{d}z),
\end{split}
\]
\[
\begin{split}
I_{2}&:=\int_{\{\delta<j_{0}(z)< 1\}}(\abs{\varphi(x_{1}+j(x_{1},z))-\varphi(x_{2}+j(x_{2},z))}\\
& \quad  +\abs{\varphi(x_{1})-\varphi(x_{2})}+\abs{j(x_{1},z)\cdot \nabla\varphi(x_{1})-j(x_{2},z)\cdot \nabla\varphi(x_{2})}) \ \nu(\mathrm{d}z),
\end{split}
\]
\[
\begin{split}
I_{3}&:=\int_{\{j_{0}(z)\geq 1\}}(\abs{\varphi(x_{1}+j(x_{1},z))-\varphi(x_{2}+j(x_{2},z))}+\abs{\varphi(x_{1})-\varphi(x_{2})}) \ \nu(\mathrm{d}z).
\end{split}
\]
Estimating $I_{1}$, we have 
\begin{equation}
\begin{split}
I_{1}&=\int_{\{j_{0}(z)\leq \delta\}}\abs{j(x_{1},z)\cdot \nabla \varphi(w_{1,z})-j(x_{1},z)\cdot \nabla\varphi(x_{1})}\\
& \quad \quad +\abs{j(x_{2},z)\cdot \nabla\varphi(w_{2,z})-j(x_{2},z)\cdot \nabla\varphi(x_{2})}\nu(\mathrm{d}z)\\
&\leq \int_{\{j_{0}(z)\leq \delta\}} j_{0}(z)\abs{\nabla\varphi(w_{1,z})-\nabla\varphi(x_{1})}+j_{0}(z)\abs{\nabla\varphi(w_{2,z})-\nabla\varphi(x_{2})}\nu(\mathrm{d}z)\\
&\leq  \norm{\varphi}_{C^{1,\alpha}(\overline{\Omega^{1}})}\left(\int_{\{j_{0}(z)\leq \delta\}}j_{0}(z)\abs{w_{1,z}-x_{1}}^{\alpha}\nu(\mathrm{d}z)+\int_{\{j_{0}(z)\leq \delta\}}j_{0}(z)\abs{w_{2,z}-x_{2}}^{\alpha}\nu(\mathrm{d}z)\right)\\
&\leq 2 \norm{\varphi}_{C^{1,\alpha}(\overline{\Omega^{1}})} \int_{\{j_{0}(z)\leq \delta\}}[j_{0}(z)]^{1+\alpha}\nu(\mathrm{d}z)\\
&\leq 2\norm{\varphi}_{C^{1,\gamma}(\overline{\Omega^{1}})}\delta^{2\alpha-\gamma}\int_{\{j_{0}(z)<1\}}[j_{0}(z)]^{\gamma+1-\alpha}\nu(\mathrm{d}z) \\
&\leq  2C_{0}\norm{\varphi}_{C^{1,\alpha}(\overline{\Omega^{1}})}\abs{x_{1}-x_{2}}^{\frac{2\alpha-\gamma}{2}}, 
\end{split}
\end{equation}
for some $w_{1,z},w_{2,z}$ satisfying $\abs{w_{1,z}-x_{1}}\leq \abs{j(x_{1},z)}$ and $\abs{w_{2,z}-x_{2}}\leq \abs{j(x_{2},z)}$.  Estimating $I_{2}$, we have
\begin{equation}
\begin{split}
I_{2}&\leq \int_{\{\delta < j_{0}(z)< 1\}}C_{\varphi}\abs{x_{2}+j(x_{2},z)-(x_{1}+j(x_{1},z))}+C_{\varphi}\abs{x_{1}-x_{2}}\\
& +\abs{j(x_{1},z)\cdot \nabla\varphi(x_{1})-j(x_{2},z)\cdot \nabla \varphi(x_{2})}\nu(\mathrm{d}z)\\
& \leq \abs{x_{1}-x_{2}}\int_{\{\delta <j_{0}(z)< 1\}}(2C_{\varphi}+C_{\varphi}C_{j}(z))\nu(\mathrm{d}z)\\
& +\int_{\{\delta< j_{0}(z)< 1\}}\abs{j(x_{1},z)\cdot (\nabla\varphi(x_{1})-\nabla\varphi(x_{2}))+j(x_{1},z)\cdot \nabla\varphi(x_{2})-j(x_{2},z)\cdot \nabla\varphi(x_{2})}\nu(\mathrm{d}z)\\
& \leq 2C_{0}C_{\varphi}\abs{x_{1}-x_{2}}\delta^{-\gamma}+C_{\varphi}\int_{\mathbb{R}^{l}}C_{j}(z)\nu(\mathrm{d}z)\abs{x_{1}-x_{2}}\\
& +\norm{\varphi}_{C^{1,\alpha}(\overline{\Omega^{1}})}\int_{\{\delta< j_{0}(z)< 1\}}\abs{x_{1}-x_{2}}^{\alpha}\nu(\mathrm{d}z)+C_{\varphi}\abs{x_{1}-x_{2}}\int_{\mathbb{R}^{l}}C_{j}(z)\nu(\mathrm{d}z)\\
& \leq 2C_{0}C_{\varphi}\abs{x_{1}-x_{2}}\delta^{-\gamma}+C_{\varphi}\int_{\mathbb{R}^{l}}C_{j}(z)\nu(\mathrm{d}z)\abs{x_{1}-x_{2}}+C_{0}\norm{\varphi}_{C^{1,\gamma}(\overline{\Omega^{1}})}\abs{x_{1}-x_{2}}^{\alpha}\delta^{-\gamma}\\
&  +\int_{\mathbb{R}^{l}}C_{j}(z)\nu(\mathrm{d}z)\abs{x_{1}-x_{2}}\\
&\leq C\abs{x_{1}-x_{2}}^{\frac{2\alpha-\gamma}{2}}+C\abs{x_{1}-x_{2}}+C\norm{\varphi}_{C^{1,\gamma}(\overline{\Omega^{1}})}\abs{x_{1}-x_{2}}^{\frac{2\alpha-\gamma}{2}}+C\abs{x_{1}-x_{2}}\\
&\leq C\abs{x_{1}-x_{2}}^{\frac{2\alpha-\gamma}{2}}.
\end{split}
\end{equation}
We briefly remark about the last two inequalities above.  Let $\mathrm{diam}(\Omega):=\max_{x,y\in \Omega}\abs{x-y}$.  If $\delta=\abs{x_{1}-x_{2}}^{\frac{1}{2}}$, we have $\abs{x_{1}-x_{2}}\delta^{-\gamma}= \abs{x_{1}-x_{2}}^{1-\frac{\gamma}{2}}\leq \abs{x_{1}-x_{2}}^{\frac{2\alpha-\gamma}{2}}$ along with $\abs{x_{1}-x_{2}}^{\alpha}\delta^{-\gamma}= \abs{x_{1}-x_{2}}^{\frac{2\alpha-\gamma}{2}}$.  If instead $\delta=1<\abs{x_{1}-x_{2}}^{\frac{1}{2}}$, then we have $\abs{x_{1}-x_{2}}\delta^{-\gamma}\leq C\abs{x_{1}-x_{2}}^{\frac{2\alpha-\gamma}{2}}$ with $C=(\mathrm{diam}(\Omega))^{\frac{2-2\alpha+\gamma}{2}}$ along with $\abs{x_{1}-x_{2}}^{\alpha}\leq C\abs{x_{1}-x_{2}}^{\frac{2\alpha-\gamma}{2}}$ with $C=(\mathrm{diam}(\Omega))^{\frac{\gamma}{2}}$.  Estimating $I_{3}$, we find
\begin{equation}
\begin{split}
I_{3}&\leq \int_{\{j_{0}(z)>1\}}C_{\varphi}\left(\abs{x_{2}-x_{1}+j(x_{2},z)-j(x_{1},z)}+\abs{x_{1}-x_{2}}\right)\nu(\mathrm{d}z)\\
&\leq \abs{x_{1}-x_{2}}\int_{\{j_{0}(z)>1\}}C_{\varphi}\left(2+C_{j}(z)\right)\nu(\mathrm{d}z)\\
&\leq C\abs{x_{1}-x_{2}}^{\frac{2\alpha-\gamma}{2}}, \text{with} \ C=(\mathrm{diam}(\Omega))^{\frac{2-2\alpha+\gamma}{2}}.
\end{split}
\end{equation}
Combining these estimates for $I_{1},I_{2},I_{3}$, we have
\begin{equation}
\begin{split}
&\abs{I\varphi(x_{1})-I\varphi(x_{2})} \leq C\abs{x_{1}-x_{2}}^{\frac{2\alpha-\gamma}{2}},
\end{split}
\end{equation}
for $C$ independent of $x_{1},x_{2}$.

\subsubsection*{Proof of Lemma~\ref{moment-est}}

Set $Y_{t}=X_{t}-X^{\epsilon}_{t}$ and apply It\^{o}'s formula with $\varphi(y,t)=\abs{y}^{2}e^{-\alpha t}$ to obtain
\begin{equation}
\label{ito-moment-estimate}
\begin{split}
\mathrm{d}\varphi(Y_{t},t)=a_{t}\mathrm{d}t+\sum_{k=1}^{d}b^{k}_{t} \ \mathrm{d}W^{k}_{t}+\int_{\mathbb{R}^{l}}c(t,z)\tilde{N}(\mathrm{d}t,\mathrm{d}z),
\end{split}
\end{equation}
where $a_{t}$, $b^{k}_{t}$, and $c(t,z)$ can be obtained from their counterparts in \eqref{eq:dY} by replacing $X^2$ by $X^{\epsilon}$.
%\begin{equation*}
%\begin{split}
%a_{t}&:=\partial_{t}\varphi(Y_{t},t)+\sum_{i=1}^{n}[\tilde{b}_{i}(X_{t})-\tilde{b}_{i}(X^{\epsilon}_{t})]\partial_{i}\varphi(Y_{t},t)\\
%&+\frac{1}{2}\sum_{i,j=1}^{n}\Bigg(\sum_{k=1}^{d}[\sigma_{ik}(X_{t})-\sigma_{ik}(X^{\epsilon}_{t})][\sigma_{jk}(X_{t})-\sigma_{jk}(X^{\epsilon}_{t})]\Bigg)\partial^{2}_{ij}\varphi(Y_{t},t)\\
%&+\int_{\mathbb{R}^{l}}[\varphi(Y_{t}+j(X_{t},z)-j^{\epsilon}(X^{\epsilon}_{t},z),t)-\varphi(Y_{t},t)\\
%&-\sum_{i=1}^{n}[j_{i}(X_{t},z)-j^{\epsilon}_{i}(X^{\epsilon}_{t},z)]\partial_{i}\varphi(Y_{t},t)\nu(\mathrm{d}z),\\
%b^{k}_{t}&:=\sum_{i=1}^{n}(\sigma_{ik}(X_{t})-\sigma_{ik}(X^{\epsilon}_{t}))\partial_{i}\varphi(Y_{t},t),\\
%c(t,z)&:= \varphi(Y_{t}+j(X_{t},z)-j^{\epsilon}(X^{\epsilon}_{t},z),t)-\varphi(Y_{t},t),
%\end{split}
%\end{equation*}
Also recall \eqref{eq:phider}. From above, we know    
\begin{equation}
\begin{split}
\varphi(y+\tilde{j}(z,t),t)-\varphi(y,t)-\sum_{i}\tilde{j}_{i}(z,t)\partial_{i}\varphi(y,t)=\abs{\tilde{j}(z,t)}^{2}e^{-\alpha t},
\end{split}
\end{equation}
with $\tilde{j}(z,t):=j(X_{t},z)-j^{\epsilon}(X^{\epsilon}_{t},z)$.  Using the fact that for each $\varepsilon>0$, there exists a $C_{\varepsilon}>0$ such that $(a+b)^{2}\leq (1+\varepsilon)a^{2}+(1+C_{\varepsilon})b^{2}$, we find
\begin{equation}
\begin{split}
&\int_{\mathbb{R}^{l}}\abs{j(X_{t},z)-j^{\epsilon}(X^{\epsilon}_{t},z)}^{2}e^{-\alpha t}\nu(\mathrm{d}z)\\
&\leq (1+\varepsilon)\int_{\mathbb{R}^{l}}\abs{j(X_{t},z)-j(X^{\epsilon}_{t},z)}^{2}e^{-\alpha t}\nu(\mathrm{d}z)\\
&\quad +(1+C_{\varepsilon})\int_{\mathbb{R}^{l}}\abs{j(X^{\epsilon}_{t},z)-j^{\epsilon}(X^{\epsilon}_{t},z)}^{2}e^{-\alpha t}\nu(\mathrm{d}z)\\
&\leq(1+\varepsilon)\beta_{j} \abs{X_{t}-X^{\epsilon}_{t}}^{2}e^{-\alpha t}+(1+C_{\varepsilon})e^{-\alpha t}\norm{j-j^{\epsilon}}_{0,2}^{2}.
\end{split}
\end{equation}
With this estimate, we find
\begin{equation*}
\begin{split}
a_{t}\leq \Big [-\alpha +\beta+\varepsilon \beta_{j}\Big ]\varphi(Y_{t},t)+(1+C_{\varepsilon})e^{-\alpha t}\norm{j-j^{\epsilon}}_{0,2}^{2}.
\end{split}
\end{equation*}
Using this inequality and taking expectations in (\ref{ito-moment-estimate}) yields
\begin{equation}
\label{est-01}
\begin{split}
\mathbb{E}\Bigg [\Big [\alpha-\beta-\varepsilon \beta_{j}\Big ] \int_{0}^{t}\abs{X_{t}-X^{\epsilon}_{t}}^{2}e^{-\alpha t}\mathrm{d}s + \abs{X_{t}-X^{\epsilon}_{t}}^{2}e^{-\alpha t} \Bigg ] \leq \frac{1+C_{\varepsilon}}{\alpha}\norm{j-j^{\epsilon}}_{0,2}^{2}.
\end{split}
\end{equation}
Recall, the following stochastic integral inequalities (see e.g. \cite{M-2008}).  For any $p>0$, there is a constant $C_{p}>0$ (in particular, $C_{1}=3$, $C_{2}=4$) such that 
\begin{equation}
\label{stocheq1}
\begin{split}
\mathbb{E}\Bigg [\underset{0\leq r \leq t} \sup \abs{\int_{0}^{r}f(s)\mathrm{d}W_{s}}^{p}\Bigg ] \leq C_{p} \ \mathbb{E} \Bigg [ \left(\int_{0}^{t}\abs{f(s)}^{2}\mathrm{d}s \right)^{p/2}\Bigg ],   
\end{split}
\end{equation}
and for the stochastic Poisson integral, if $0<p\leq 2$, then 
\begin{equation}
\label{stocheq2}
\begin{split}
\mathbb{E}\Bigg [\underset{0\leq r \leq t} \sup \abs{\int_{\mathbb{R}^{l}\times (0,r)}g(s,z)\tilde{N}(\mathrm{d}z, \mathrm{d}s)}^{p} \Bigg ]  \leq C_{p} \ \mathbb{E} \Bigg [ \left(\int_{0}^{t}\mathrm{d}s\int_{\mathbb{R}^{l}}\abs{g(s,z)}^{2}\nu(\mathrm{d}z)\right)^{p/2}\Bigg ] .
\end{split}
\end{equation}  
Now, coming back to (\ref{ito-moment-estimate}) to take first the supremum and then the expectation, we deduce after using (\ref{stocheq1}), (\ref{stocheq2}) with $p=1$,
\begin{equation}
\label{sup-estimate}
\begin{split}
\mathbb{E}\Bigg [ \underset{0\leq s \leq t}\sup \abs{X_{s}-X^{\epsilon}_{s}}^{2}e^{-\alpha t}\Bigg ] \leq 3 \ \mathbb{E}\Bigg [ \left (\sum_{k}\int_{0}^{t}\abs{b^{k}_{s}}^{2}\mathrm{d}s\right )^{1/2}
+\left ( \int_{0}^{t}\mathrm{d}s \int_{\mathbb{R}^{l}}\abs{c(s,t)}^{2}\nu(\mathrm{d}z)\right )^{1/2}\Bigg ].
\end{split}
\end{equation}
We now estimate the two terms on the right hand side of the above inequality.  First, for some $C$ depending on the Lipschitz constant $C_{\sigma}$ in (H1), we have that $\sum_{k}\abs{b^{k}_{s}}^{2}\leq C \abs{\varphi(Y_{s},s)}^{2}$ by the following inequalities
\begin{equation*}
\begin{split}
\sum_{k}\abs{b^{k}_{s}}^{2}&=\sum_{k}\abs{\sum_{i}(\sigma_{ik}(X_{t})-\sigma_{ik}(X_{t}^{\epsilon}))\frac{2(X_{i}(t)-X^{\epsilon}_{i}(t))}{\abs{X_{t}-X^{\epsilon}_{t}}^{2}}\varphi(Y_{s},s)}^{2}\\
&\leq 4n\frac{\abs{\varphi(Y_{s},s)}^{2}}{\abs{X_{t}-X_{t}^{\epsilon}}^{4}}\sum_{k,i}\left(\sigma_{ik}(X_{t})-\sigma_{ik}(X_{t}^{\epsilon})\right)^{2}\left(X_{i}(t)-X_{i}^{\epsilon}(t)\right)^{2}\\
&\leq 2n\frac{\abs{\varphi(Y_{s},s)}^{2}}{\abs{X_{t}-X_{t}^{\epsilon}}^{4}}\left(\sum_{k,i}\left(\sigma_{ik}(X_{t})-\sigma_{ik}(X_{t}^{\epsilon})\right)^{4}+\sum_{k,i}\left(X_{i}(t)-X_{i}^{\epsilon}(t)\right)^{4}\right)\\
&\leq 2n\frac{\abs{\varphi(Y_{s},s)}^{2}}{\abs{X_{t}-X_{t}^{\epsilon}}^{4}}\left(C_{\sigma}^{4}\abs{X(t)-X^{\epsilon}(t)}^{4}+d\abs{X(t)-X^{\epsilon}(t)}^{4}\right)\\
&\leq 2n(C_{\sigma}^{4}+d)\abs{\varphi(Y_{s},s)}^{2}.
\end{split}
\end{equation*}
Using the above, we now have
\begin{equation*}
\begin{split}
\mathbb{E}\Bigg [ \left(\sum_{k}\int_{0}^{t}\abs{b^{k}_{s}}^{2}\mathrm{d}s\right)^{1/2}\Bigg ] &\leq C \ \mathbb{E}\Bigg [ \left (\underset{0\leq s \leq t}\sup \abs{\varphi(Y_{s},s)}\right )^{1/2}\left ( \int_{0}^{t}\abs{\varphi(Y_{s},s)}\mathrm{d}s\right )^{1/2}\Bigg ].
\end{split}
\end{equation*}
Thus, by means of the inequality $2ab \leq \varepsilon a^{2}+b^{2}/\varepsilon$ and the H\"{o}lder inequality we deduce that 
\begin{equation}
\label{b-est}
\begin{split}
3\mathbb{E}\Bigg [ \left(\sum_{k}\int_{0}^{t}\abs{b^{k}_{s}}^{2}\mathrm{d}s\right )^{1/2}\Bigg ] \leq \frac{1}{3}\mathbb{E}\Bigg[\underset{0\leq s\leq t}\sup \abs{\varphi(Y_{s},s)}\Bigg]+C_{1}\mathbb{E}\Bigg [ \int_{0}^{t}\abs{\varphi(Y_{s},s)}\mathrm{d}s\Bigg ].
\end{split}
\end{equation}
The term corresponding to Poisson integral can be handled using the same technique.  Towards this end, note that
\begin{equation}
\begin{split}
\abs{c(s,z)}^{2}\leq \abs{j(X_{s},z)-j^{\epsilon}(X^{\epsilon}_{s},z)}^{2}\int_{0}^{1}\abs{\nabla\varphi(Y_{s}+\theta(j(X_{s},z)-j^{\epsilon}(X_{s}^{\epsilon},z),s))}^{2}\mathrm{d}\theta.
\end{split}
\end{equation}
Estimating the gradient $\nabla \varphi$ and using $y:=X_{s}-X^{\epsilon}_{s}$ to ease notation, we have
\begin{equation}
\label{grad}
\begin{split}
\abs{\nabla \varphi(y+\theta \tilde{j},s)}^{2}= 4\varphi(y+\theta \tilde{j},s) e^{-\alpha s}=4e^{-2\alpha s}\abs{y+\theta \tilde{j}}^{2} \leq 8e^{-2\alpha s}(\abs{y}^{2}+\abs{\tilde{j}}^{2}).
\end{split}
\end{equation}
Thus, we know $\abs{c(s,z)}^{2}\leq 8e^{-2\alpha s}\abs{\tilde{j}}^{2}\left(\abs{y}^{2}+\abs{\tilde{j}}^{2}\right).$
Now, assuming  $C_{j}(z) \in L^{4}(\mathbb{R}^{l})$, we have for $p=2,4$
\begin{equation*}
\begin{split}
\int_{\mathbb{R}^{l}}\abs{j(X_{s},z)-j^{\epsilon}(X^{\epsilon}_{s},z)}^{p}\nu(\mathrm{d}z)&\leq 2^{p-1}\norm{j-j^{\epsilon}}_{0,p}^{p}+2^{p-1}\int_{\mathbb{R}^{l}}\abs{j^{\epsilon}(X_{s},z)-j^{\epsilon}(X^{\epsilon}_{s},s)}^{p}\nu(\mathrm{d}z)\\
&\leq 2^{p-1}\norm{j-j^{\epsilon}}_{0,p}^{p}+2^{p-1}\abs{X_{s}-X^{\epsilon}_{s}}^{p}\int_{\mathbb{R}^{l}}[C_{j}(z)]^{p}\nu(\mathrm{d}z)\\
&\leq C\norm{j-j^{\epsilon}}_{0,p}^{p}+C\abs{X_{s}-X^{\epsilon}_{s}}^{p}.\\
\end{split}
\end{equation*} 
Using this estimate and the inequality $ab\leq \frac{a^{p}}{p}+\frac{b^{q}}{q}$ for $1/p+1/q=1$, the following holds
\begin{equation}
\begin{split}
\int_{\mathbb{R}^{l}}\abs{c(s,z)}^{2}\nu(\mathrm{d}z) &\leq  \int_{\mathbb{R}^{l}}8e^{-2\alpha s}\abs{\tilde{j}}^{2}\left(\abs{y}^{2}+\abs{\tilde{j}}^{2}\right)\nu(\mathrm{d}z)\\
&\leq 8 e^{-2\alpha s}\abs{y}^{2}\left(C\norm{j-j^{\epsilon}}_{0,2}^{2}+C\abs{y}^{2}\right) + 8 e^{-2\alpha s}\left(C\norm{j-j^{\epsilon}}_{0,4}^{4}+C\abs{y}^{4}\right)\\
&\leq C\abs{\varphi(Y_{s},s)}^{2}+Ce^{-2\alpha s}\left(\norm{j-j^{\epsilon}}_{0,4}^{4}+\norm{j-j^{\epsilon}}_{0,2}^{4}\right)\\
&\leq C\abs{\varphi(Y_{s},s)}^{2} + Ce^{-2\alpha s}\Lambda_{0,2}^{4}(j-j^{\epsilon}).
\end{split}
\end{equation}
Returning back to (\ref{sup-estimate}) and using $(a+b)^{p}\leq a^{p} + b^{p}$ for $0<p<1$, we find
\begin{equation}
\begin{split}
\mathbb{E} \Bigg [ \left (\int_{0}^{t}\mathrm{d}s \int_{\mathbb{R}^{l}}\abs{c(s,t)}^{2}\nu(\mathrm{d}z)\right )^{1/2}\Bigg ] &\leq \mathbb{E} \Bigg [ \left ( \int_{0}^{t} C\abs{\varphi(Y_{s},s)}^{2} + Ce^{-2\alpha s}\Lambda_{0,2}^{4}(j-j^{\epsilon})\mathrm{d}s \right )^{1/2}\Bigg ]\\
&\leq \mathbb{E} \Bigg [ \left(\int_{0}^{t}C\abs{\varphi(Y_{s},s)}^{2} \mathrm{d}s \right)^{1/2} \Bigg ]+ C \Lambda_{0,2}^{2}(j-j^{\epsilon}).
\end{split}
\end{equation}
The first term can be handled in the same manner as the Weiner term above to yield an estimate as in (\ref{b-est}).  Now, combining these two estimates, referring back to (\ref{sup-estimate}), and using (\ref{est-01}), we conclude 
\begin{equation}
\begin{split}
\mathbb{E}\Bigg [\underset{0\leq s \leq t}\sup\abs{X_{s}-X^{\epsilon}_{s}}^{2} e^{-\alpha s}\Bigg ] \leq C\Lambda_{0,2}^{2}(j-j^{\epsilon}).
\end{split}
\end{equation}

\newpage

\bibliographystyle{siam}

\bibliography{biblio}
\end{document}